\documentclass[11pt,a4paper]{paper}
 
\pdfoutput=1 
\usepackage{amsmath}
\usepackage[titletoc]{appendix}
\usepackage{enumerate}
\usepackage{multicol}
\usepackage{amsthm}
\usepackage{listings} 
\usepackage[top=2cm, bottom=3cm, left=2.5cm, right=2.5cm]{geometry}
\usepackage{booktabs}
\usepackage{subfigure}
\usepackage{fancyhdr}
\usepackage{epsfig}
\usepackage{algorithm}
\usepackage{algorithmic}
\usepackage{nomencl}
\usepackage[authoryear]{natbib}
\usepackage[titletoc]{appendix}
\usepackage[T1]{fontenc}
\usepackage[utf8]{inputenc} 
\usepackage{leftidx}
\usepackage{setspace}
\makenomenclature

\usepackage{fourier}

\usepackage[english]{babel}															% English language/hyphenation
\usepackage{amsfonts}										% Math packages
\newtheorem{theoremm}{Theorem}

\newtheorem{lma}{Lemma}

\newcommand{\E}{\operatorname{E}}
\newcommand{\pr}{\operatorname{Pr}}

\newcommand{\var}{\operatorname{Var}}

\newcommand{\diag}{\operatorname{diag}}
\newcommand{\ba}{\mathbf{A}}
\newcommand{\bb}{\mathbf{B}}
\newcommand{\bc}{\mathbf{C}}
\newcommand{\bd}{\mathbf{D}}

\newcommand{\bh}{\mathbf{H}}

\newcommand{\bk}{\mathbf{K}}

\newcommand{\bp}{\mathbf{P}}

\newcommand{\bs}{\mathbf{S}}

\newcommand{\bw}{\mathbf{W}}

\newcommand{\by}{\mathbf{y}}

\newcommand{\sa}{\mathbf{a}}
\newcommand{\ssc}{\mathbf{c}} % %DIFFERENT!!!!!!!!
\newcommand{\sd}{\mathbf{d}}

\newcommand{\ssf}{\mathbf{f}}

\newcommand{\ssl}{\mathbf{l}} % %DIFFERENT!!!!!!!!

\newcommand{\st}{\mathbf{t}}

\newcommand{\sw}{\mathbf{w}}
\newcommand{\sx}{\mathbf{x}}
\newcommand{\sz}{\mathbf{z}}

\newcommand{\btheta}{\boldsymbol{\theta}}
\newcommand{\bomega}{\boldsymbol{\omega}}
\newcommand{\blambda}{\boldsymbol{\lambda}}

\newcommand{\bDelta}{\boldsymbol{\Delta}}
\newcommand{\bbeta}{\boldsymbol{\beta}}
\newcommand{\bmu}{\boldsymbol{\mu}}
\newcommand{\bUpsilon}{\boldsymbol{\Upsilon}}

\newcommand{\bSigma}{\boldsymbol{\Sigma}}
\newcommand{\bOmega}{\boldsymbol{\Omega}}

\newcommand{\bzero}{\boldsymbol{0}}
\newcommand{\bone}{\boldsymbol{1}}
\newcommand{\bxi}{\boldsymbol{\xi}}
\newcommand{\bpsi}{\boldsymbol{\psi}}

\newcommand{\hbomega}{\hat{\bomega}}

\newcommand{\sxij}{\sx_{ij}}
\newcommand{\sxbisw}{\bar{\sx}_{i}^{(w)T}}

\newcommand{\psibb}{\psi^{(b)}}
\newcommand{\psiww}{\psi^{(w)}}
\newcommand{\bpsibb}{\bpsi^{(b)}}
\newcommand{\bpsiww}{\bpsi^{(w)}}

\newcommand{\bxibb}{\boldsymbol{\xi}^{(b)}}
\newcommand{\bxiww}{\boldsymbol{\xi}^{(w)}}

\newcommand{\dbeta}{\dot{\beta}}
\newcommand{\sige}{\sigma_e}
\newcommand{\siga}{\sigma_{\alpha}}

\newcommand{\dsiga}{\dot{\sigma}_{\alpha}}
\newcommand{\dsige}{\dot{\sigma}_e}

\newcommand{\yij}{y_{ij}}
\newcommand{\eij}{e_{ij}}

\newcommand{\yib}{\bar{y}_{i}}

\newcommand{\swy}{S_{w}^{y}}
\newcommand{\sumig}{\sum_{i=1}^g}

\newcommand{\sumjmi}{\sum_{j=1}^{m_i}}

\newcommand{\bkmh}{\bk^{-1/2}}
\newcommand{\bkh}{\bk^{1/2}}
\begin{document}

\title{Increasing Cluster Size Asymptotics for Nested Error Regression Models}
%\tnotetext[mytitlenote]{Fully documented templates are available in the elsarticle package on \href{http://www.ctan.org/tex-archive/macros/latex/contrib/elsarticle}{CTAN}.}

\author{
\normalfont 								\normalsize
    Ziyang Lyu and A.H. Welsh \\
    \normalfont 								\normalsize
   Mathematical Sciences Institute and Research School of Finance, Actuarial Studies and Statistics \\
    \normalfont 								\normalsize
        Australian National University}

\date{today}

\maketitle

\abstract{
This paper establishes asymptotic results for the maximum likelihood and restricted maximum likelihood (REML) estimators of the parameters in the nested  error regression model for clustered data when both of the number of independent clusters and the cluster sizes (the number of observations in each cluster) go to infinity.  Under very mild conditions, the estimators are shown to be asymptotically normal with an elegantly structured covariance matrix.  There are no restrictions on the rate at which the cluster size tends to infinity but it turns out that we need to treat within cluster parameters (i.e. coefficients of unit-level covariates that vary within clusters and the within cluster variance) differently from between cluster parameters (i.e. coefficients of cluster-level covariates that are constant within clusters and the between cluster variance) because they require different normalisations and are asymptotically independent.
}


%\MSC[2010] 62E20,  62J05

\thispagestyle{empty}

\section{Introduction}

Regression models with nested errors (also called random intercept or homogeneous correlation models) are widely used in applied statistics to model relationships in clustered data; they were introduced for survey data, by \cite{scott1982effect} and \cite{battese1988error}, and, for longitudinal data, by \cite{laird1982random}.  The models are usually fitted (see \cite{harville1977maximum}) by assuming normality and computing maximum likelihood or restricted maximum likelihood (REML) estimators.  As these estimators are nonlinear, asymptotic results provide an important way to understand their properties and then to construct approximate inferences about the unknown parameters.  The usual asymptotic results applied to these estimators from \cite{hartley1967maximum}, \cite{anderson1969statistical}, \cite{miller1977asymptotic}, \cite{das1979asymptotic}, \cite{cressie1993asymptotic}, and \cite{richardson1994asymptotic} increase the number of clusters while keeping the size of each cluster fixed or bounded.  However, there are many applications, particularly with survey data, with large cluster sizes; for example, \cite{arora1997} give an example with $43$ clusters and cluster sizes ranging from $95$ to $633$ and such examples are common in analysing poverty data \citep{pratesi2016poverty}.  In addition, there are theoretical problems (e.g. in prediction, see \cite{jiang1998asymptotic}) for which both the number of clusters and the cluster sizes need to increase.  Therefore, in this paper, we study the asymptotic properties of normal-theory maximum likelihood and REML estimators of the parameters in the nested error regression model as both the number of clusters and the cluster sizes tend to infinity.

Suppose that we observe on the $j$th unit in the $i$th cluster the vector  $[y_{ij}, \sx_{ij}^{T}]^T$, where $y_{ij}$ is a scalar response variable and $\sx_{ij}$ is a vector of explanatory variables or covariates, $j=1,\ldots,m_i$, $i=1,\ldots, g$. The nested error regression model specifies that
\begin{equation}\label{nerm}
\yij= \beta_0+ \sx_{ij}^T\bbeta_s +\alpha_i+\eij, \qquad j=1,\ldots,m_i, \, i=1,\ldots,g,
\end{equation}
where $\beta_0$ is the intercept, $\bbeta_s$ is the slope parameter,  $\alpha_i$ is a random effect representing a random cluster effect and $e_{ij}$ is an error term.  We assume that the $\{\alpha_i\}$ and $\{e_{ij}\}$ are all mutually independent with mean zero and variances (called the variance components) $\siga^2$ and  $\sige^2$, respectively; we do not assume normality.   This regression model treats clusters as independent with constant (i.e. homogeneous) correlation within clusters.  It is a particular, simple linear mixed model that is widely used in fields such as small area estimation (see \cite{rao2014small})  to model and make predictions from clustered data, so our results are immediately useful.  In addition, its simplicity allows us to use elementary methods to gain insight into exactly what is going on and obtain explicit, highly interpretable results as the cluster sizes increase.  These arguments and results form the basis for how to proceed to more complicated cases, with multiple variance components.

When the random effects and errors are normally distributed, the likelihood for the parameters and the REML criterion can be obtained analytically.  Irrespective of whether normality holds or not, we refer to these functions as the likelihood and the REML criterion for the model (\ref{nerm}) and the values of the parameters that maximise them as maximum likelihood and REML estimators, respectively.  For our results, we make very simple assumptions: essentially finite ``$4+\delta$'' moments for the random effects and errors (instead of normality) and, allowing the explanatory variables to be fixed or random, conditions analogous to finite ``$2+\delta$'' moments for the explanatory variables. We allow $g \to \infty$ and $\min_{1\le i \le g} m_i \to \infty$ without any restriction on the rates.  We obtain asymptotic representations for both the maximum likelihood and REML estimators that give the influence functions of these estimators, are very useful for deriving results when we combine these estimators with other estimators, and lead to central limit theorems for these estimators and asymptotic inferences for the unknown parameters.   The normalisation is by a diagonal matrix which is easy to interpret.  These results provide new and striking insights. First, we need to separate and treat within cluster parameters (i.e. coefficients of unit-level covariates that vary within clusters and the within cluster variance $\sigma_e^2$) differently from between cluster parameters (i.e. coefficients of cluster-level covariates that are constant within clusters and the between cluster variance $\sigma_{\alpha}^2$).   We make explicit the fact that the information for within cluster parameters grows with $n = \sum_{i=1}^g m_i$ and the information for between cluster parameters grows with $g$ so they require different normalisations.  The asymptotic variance matrix which we obtain explicitly has a very tidy and easy to interpret block diagonal structure.  Second, there are good reasons for centering the within cluster covariates about their cluster means and then including the cluster means as contextual effect variables in the between cluster covariates (see for example \cite{yoon2020effect} for references) but our asymptotic results (which include both cases) show that increasing cluster size has asymptotically the same effect as the centering (although without increasing the number of between cluster parameters) and also asymptotically orthogonalises the variance components. These apparently simple insights are new and not available from the existing literature.  

The few results in the literature that allow both the number of clusters and the cluster size to go to infinity do not give the same insights as our results.  \cite{jiang1996reml} proved consistency and asymptotic normality of the maximum likelihood and REML estimators for a wide class of linear mixed models allowing increasing cluster sizes.  He later showed this condition is required for studying the empirical distribution of the empirical best linear unbiased predictors (EBLUPs) of the random effects \citep{jiang1998asymptotic}.  % \cite{jiang1996reml} went beyond the scope of the present paper by also allowing the number of covariates to increase with the sample size, a case in which the asympto.     
\cite{xie2003asymptotics} obtained results for generalized estimating equation regression parameter estimators with increasing cluster size which potentially relate to our estimators, but their estimators do not include the variance components so the results do not apply to our estimators.  The difficulties with trying to apply general results to particular models like (\ref{nerm}) are that it can be difficult to understand the conditions and interpret the main result.   To illustrate, increasing cluster size in \cite{jiang1996reml} is a part of other complicated assumptions and, for particular examples, he needed further conditions on the way the cluster size increases, making it difficult to see whether there is any restriction on the relationship between the cluster size and the number of clusters and leaving open questions of whether the conditions are minimal or not.  Also, although Jiang did give some nested model examples which satisfy his main invariant class $AI^4$ condition, this condition is quite complicated. In terms of their main results, both \cite{jiang1996reml} and \cite{xie2003asymptotics} normalise the estimators by the product of general (nondiagonal) matrices, producing results which are difficult to interpret and do not provide the insights our results provide.

We introduce notation to describe the maximum likelihood and REML estimators for the parameters in  (\ref{nerm})-(\ref{mean}), specify the conditions and state our main results in Section \ref{sec:not}.  We discuss the results in Section \ref {sec:disc} and give the proofs in Section  \ref{sec:proof}. 

%%%%%%%%%%%%%%%%%%%%%%%%%%%%%%%%%%%%%%%%%%%%%%%%%%%%%

\section{Results} \label{sec:not}

We gain important insights by partitioning the vector of covariates $\sx_{ij}$ into the $p_w$-vector $\sx_{ij}^{(w)}$ of within cluster covariates and the $p_b$-vector $\sx_{i}^{(b)}$ of between cluster covariates.  As noted in the Introduction, it is also often useful to center the within cluster covariates about their cluster means and then expand the between cluster covariate vector to include the cluster means of the within cluster covariates.    Specifically, for a single within cluster covariate $x_{ij}$, we can make the regression function either $\beta_0+ x_{ij}\beta_2$ or the centered form $\beta_0+ \bar{x}_i\beta_1 + (x_{ij}- \bar{x}_i)\beta_2$.   This centering ensures that $\sum_{j=1}^{m_i} \sx_{ij}^{(w)} = \bzero_{[p_w:1]}$ for all $i=1,\ldots, g$, where $\bzero_{[p:q]}$ denotes the $p \times q$ matrix of zeros, and as it orthogonalises the between and within covariates, has advantages for interpreting and fitting the model \citep{yoon2020effect} as well as increasing flexibility.   We leave this as choice for the modeller; our analysis handles both cases as well as the cases in which there are no within cluster or no between cluster covariates because they are all special cases of the model (\ref{nerm}) which we re-express as
\begin{equation}\label{mean}
y_{ij} = \beta_0+\sx_{i}^{(b)T}\bbeta_1+ \sx_{ij}^{(w)T}\bbeta_2 + \alpha_i + e_{ij}, \qquad j=1,\ldots,m_i, \, i=1,\ldots,g,
\end{equation}
where $\beta_0$ is the unknown intercept, $\bbeta_1$ is the unknown between cluster slope parameter and $\bbeta_2$ is the unknown within cluster slope parameter.  We treat the covariates as fixed; when they are random, we condition on them, though we omit this from the notation.  We assume throughout that the true model that describes the data generating mechanism is (\ref{mean}) with general parameter $\bomega=[\bbeta_0,\bbeta_1^T,\siga^2,\bbeta_2^T,\sige^2]^T$,  true parameter $\dot\bomega=[\dbeta_0,\dot\bbeta_1^T,\dsiga^2,\dot\bbeta_2^T,\dsige^2]^T$  and  take all expectations under the true model.  %We use the subscript zero with a vector parameter and a dot above a scalar parameter to denote the true parameter value  so
The order of the parameters in $\bomega$  and $\dot\bomega$ groups the between parameters and the within parameters together and simplifies the presentation of our results.  %We work with the model (\ref{mean}) so we have both between cluster and within cluster regression parameters to estimate.   If we only have no between cluster covariates,we discard $\bbeta_1$, while if we have no within cluster covariates, we discard $\bbeta_2$.  The results for these cases can be obtained as special cases of the general results by deleting the components of vectors  and the rows and columns of matrices corresponding to the discarded parameter.  

To simplify notation, let $\tau_i=m_i/(\sigma_e^2 + m_i \sigma_a^2)$ with true value $\dot\tau_i$, $m_L=\min_{1\le i \le g} m_i$,
\begin{displaymath}
\begin{split}
& \bar{y}_{i}=\frac{1}{{m_i}}\sum_{j=1}^{m_i}y_{ij}, \quad\bar{\sx}_{i}^{(w)}=\frac{1}{m_i}\sum_{j=1}^{m_i}\sxij^{(w)},\quad
S_{w}^{y}=\sum_{i=1}^g\sum_{j=1}^{m_i}(\yij-\yib)^2, \\&
\bs_{w}^{xy}=\sum_{i=1}^g\sum_{j=1}^{m_i}(\sxij^{(w)}-\bar{\sx}_{i}^{(w)})(\yij-\yib), \quad \text{and}\quad\\&
\bs_{w}^{x}=\sum_{i=1}^g\sum_{j=1}^{m_i}(\sxij^{(w)}-\bar{\sx}_{i}^{(w)})(\sxij^{(w)}-\bar{\sx}_{i}^{(w)})^T.
\end{split}
\end{displaymath} 
The  log-likelihood for the parameters in the model (after discarding constant terms) is
\begin{equation}\label{owlogli}
\begin{split}
l(\bomega)&=\frac{1}{2}\sumig\log(\tau_i)-\frac{n-g}2\log\sige^2-\frac{1}{2\sige^2}(\swy-2\bs_{w}^{xyT}\bbeta_{2}+ \bbeta_{2}^T\bs_{w}^x\bbeta_{2})
\\&\qquad
-\frac{1}{2}\sumig \tau_i(\yib-\beta_{0}-\sx_{i}^{(b)T}\bbeta_{1}-\bar{\sx}_{i}^{(w)T}\bbeta_{2})^2.
\end{split}
\end{equation}
To  maximize $l(\bomega)$ and find the maximum likelihood estimator $\hat{\bomega}$ of $\bomega$, we differentiate (\ref{owlogli}) with respect to $\bomega$ to obtain the estimating function $\bpsi(\bomega)$ and then solve the estimating equation $\bzero_{[p_b+p_w+3:1]} = \bpsi(\bomega)$. 
The components of  $\bpsi(\bomega)$ are 
\begin{equation}\label{eqm}
\begin{split}
&l_{\beta_{0}}(\bomega)=\sumig \tau_i (\yib-\beta_{0}-\sx_{i}^{(b)T}\bbeta_{1}-\bar{\sx}_{i}^{(w)T}\bbeta_{2}),\\&
\ssl_{\bbeta_{1}}(\bomega)=\sumig\tau_i\sx_{i}^{(b)}(\yib-\beta_{0}-\sx_{i}^{(b)T}\bbeta_{1}-\bar{\sx}_{i}^{(w)T}\bbeta_{2}),\\&
l_{\siga^2}(\bomega)=-\frac{1}{2}\sumig \tau_i +\frac{1}{2}\sumig \tau_i^2 (\yib-\beta_{0}-\sx_{i}^{(b)T}\bbeta_{1}-\bar{\sx}_{i}^{(w)T}\bbeta_{2})^2, \\&
\ssl_{\bbeta_{2}}(\bomega)=\frac{1}{\sige^2}\bs_{w}^{xy}-\frac{1}{\sige^2}\bs_{w}^x\bbeta_{2}+\sumig \tau_i\bar{\sx}_{i}^{(w)}(\yib-\beta_{0}-\sx_{i}^{(b)T}\bbeta_{1}-\bar{\sx}_{i}^{(w)T}\bbeta_{2}),\\&
l_{\sige^2}(\bomega)=-\frac{1}{2}\sumig m_i^{-1}\tau_i-\frac{n-g}{2\sige^2}+\frac{1}{2\sige^4}(S_{w}^y-2\bs_{w}^{xyT}\bbeta_{2}+\bbeta_{2}^T\bs_{w}^x\bbeta_{2})\\&\qquad\qquad+\frac{1}{2}\sumig m_i^{-1}\tau_i^2(\yib-\beta_{0}-\sx_{i}^{(b)T}\bbeta_{1}-\bar{\sx}_{i}^{(w)T}\bbeta_{2})^2 .
\end{split}
\end{equation}
Let $\bpsi(\bomega)^T=[\bpsibb(\bomega)^T,\bpsiww(\bomega)^T]$, where $\bpsibb(\bomega)^T=[l_{\beta_{0}}(\bomega),\ssl_{\bbeta_{1}}(\bomega)^T,l_{\siga^2}(\bomega)]$ are the estimating functions for the between cluster parameters and $\bpsiww(\bomega)^T=[\ssl_{\bbeta_{2}}(\bomega)^T,l_{\sige^2}(\bomega)]$ are the estimating functions for the within cluster parameters.  The derivatives of the estimating functions which we write as $\nabla \bpsi(\bomega)$ and their expected values under the model are given in the Appendix.

To control the estimating function and derive the asymptotic properties of $\hbomega$ from the estimating equation, we impose the following condition. 

\medskip\noindent
\textbf{Condition A}
\begin{itemize}
	\item[1.] The model (\ref{mean}) holds with true parameters $\dot\bomega$ inside the parameter space $\Omega$.
	\item[2.] The number of clusters $g \to \infty$ and the minimum number of observations per cluster $m_L\to\infty$. 
	\item[3.]   The random variables $\{\alpha_i\}$ and $\{\eij\}$ are independent and identically distributed and 	there is a  $\delta>0$ such that $\E| \alpha_i|^{4+\delta}<\infty$ and $\E |e_{ij}|^{4+\delta}<\infty$
	for all $i=1,\ldots,g$ and $j\in\mathcal{S}_i$.
	
	\item[4.]  Suppose that the limits $ \ssc_1 = \lim_{g \rightarrow \infty} g^{-1}\sum_{i=1}^g \sx_{i}^{(b)}$, \\ $\bc_2 = \lim_{g \rightarrow \infty} g^{-1}\sum_{i=1}^g \sx_{i}^{(b)}\sx_{i}^{(b)^T}$ and $ \bc_{3} = \lim_{g\to\infty}\lim_{m_L\to\infty} n^{-1}\bs_{w}^x$ exist and the matrices $\bc_2$ and $\bc_{3}$ are positive definite.  Suppose further that $\lim_{g\to\infty}\lim_{m_L\to\infty}\frac{1}{g}\sumig|\bar{\sx}_{i}^{(w)}|^2<\infty$, and there is a $\delta >0$ such that $\lim_{g\to\infty}g^{-1}\sumig|\sx_{i}^{(b)}|^{2+\delta}<\infty$ and  \\$\lim_{g\to\infty}\lim_{m_L\to\infty}n^{-1}\sumig\sumjmi|\sx_{ij}^{(w)} - \bar\sx_{i}^{(w)}|^{2+\delta}  <\infty$.

\end{itemize}
These are very mild conditions which are often satisfied in practice.   Conditions A3 and A4 ensure that limits needed to ensure the existence of the asymptotic variance of the estimating function exist and that we can establish a Lyapounov condition and hence a central limit theorem for the estimating function.   They also ensure that minus the appropriately normalised second derivative of the estimating function converges to $\bb$ given in (\ref{bbbbb})  below.   Unlike in the case of fixed $m_i$, A4 does not involve unknown parameters through the weights $\dot\tau_i$.  

Our main result is the following theorem which we prove in Section \ref{sec:proof}.

\begin{theoremm}  \label{thm1}
	Suppose  Condition A holds. Then, as $g, m_L \to\infty$, there is a solution $\hbomega$ to  the estimating equations $\bzero_{[p_b+p_w+3:1]}=\bpsi(\bomega)$, satisfying $|\bkh(\hbomega-\dot\bomega)|=O_p(1)$, where $\bk = \diag(g,g\bone_{p_b}^T,g,n\bone_{p_w}^T,n)$ with $\bone_p$ the $p$ vector of ones.  Moreover, $\hbomega$ has the asymptotic representation
\begin{equation}\label{rep}
\bkh(\hbomega-\dot\bomega)=	\bb^{-1}\bkmh\bxi+o_p(1),
\end{equation}
where $\bb$ is given by (\ref{bbbbb}) below and $\bxi = [\xi_{\beta_0}, \bxi_{\bbeta_1}^T, \xi_{\siga^2}, \bxi_{\bbeta_2}^T, \xi_{\sige^2}]^T$ has components
	\begin{equation*}\label{mmeq1}
	\begin{split}
	&\xi_{\beta_0}=\frac{1}{\dsiga^2}\sumig\alpha_i,\qquad \bxi_{\bbeta_1}=\frac{1}{\dsiga^2}\sumig\sx_i^{(b)}\alpha_i,\qquad 
	\xi_{\siga^2}= \frac{1}{2\dsiga^4} \sumig(\alpha_i^2-\dsiga^2),\\&
	\bxi_{\bbeta_2}=\frac{1}{\dsige^2}\sumig\sumjmi(\sx_{ij}^{(w)}-\bar{\sx}_{i}^{(w)})\eij \qquad
	\mbox{and} \qquad 
	\xi_{\sige^2}=\frac{1}{2\dsige^4}\sumig\sumjmi(\eij^2-\dsige^2).
	\end{split}
	\end{equation*}
	It follows that
	\begin{equation*}
	\bkh(\hbomega-\dot\bomega) \xrightarrow{D} N(\bzero,\bc),
	\end{equation*}
	where  
	\begin{displaymath}
	\begin{split}
	\bc &=  \left[ \begin{matrix}
	\dsiga^2 d&  \dsiga^2\sd_1^T &\E\alpha_1^3&\bzero_{[1:p_w]}&0\\
	\dsiga^2\sd_1 & \dsiga^2\bd_2&\bzero_{[p_b:1]}&\bzero_{[p_b:p_w]}&\bzero_{[p_b:1]}\\
	\E\alpha_1^3 &\bzero_{[1:p_b]}&\E\alpha_1^4-\dsiga^4&\bzero_{[1:p_w]}&0\\
	\bzero_{[p_w:1]} &\bzero_{[p_w:p_b]}&\bzero_{[p_w:1]}&\dsige^2 \bc_3^{-1}&\bzero_{[p_w:1]}\\
	0&\bzero_{[1:p_b]}&0&0& \E e_{ij}^4-\dsige^4
	\end{matrix}\right]
	\end{split}
	\end{displaymath}
with $d= 1//(1- \ssc_1^T\bc_2^{-1}\ssc_1 )$, $\sd_1= -\ssc_1^T\bc_2^{-1}/(1- \ssc_1^T\bc_2^{-1}\ssc_1 )$ and $\bd_2=\bc_2^{-1} + \bc_2^{-1}\ssc_1\ssc_1^T\bc_2^{-1}/(1- \ssc_1^T\bc_2^{-1}\ssc_1)$.
%	\begin{displaymath}
%	\begin{split}
%	\bc &=  \left[ \begin{matrix}
%	\dsiga^2/(1- \ssc_1^T\bc_2^{-1}\ssc_1 )& - \dsiga^2 \ssc_1^T\bc_2^{-1}/(1- \ssc_1^T\bc_2^{-1}\ssc_1 ) &\E\alpha_1^3&\bzero_{[1:p_w]}&0\\
%	- \dsiga^2 \bc_2^{-1}\ssc_1/(1- \ssc_1^T\bc_2^{-1}\ssc_1 ) & \dsiga^2\{\bc_2^{-1} + \bc_2^{-1}\ssc_1\ssc_1^T\bc_2^{-1}/(1- \ssc_1^T\bc_2^{-1}\ssc_1)\}&\bzero_{[p_b:1]}&\bzero_{[p_b:p_w]}&\bzero_{[p_b:1]}\\
%	\E\alpha_1^3 &\bzero_{[1:p_b]}&\E\alpha_1^4-\dsiga^4&\bzero_{[1:p_w]}&0\\
%	\bzero_{[p_w:1]} &\bzero_{[p_w:p_b]}&\bzero_{[p_w:1]}&\dsige^2 \bc_3^{-1}&\bzero_{[p_w:1]}\\
%	0&\bzero_{[1:p_b]}&0&0& \E e_{ij}^4-\dsige^4
%	\end{matrix}\right].
%	\end{split}
%	\end{displaymath}
\end{theoremm}

\medskip
We now consider REML estimation.  To describe REML, we group the parameters into the regression parameters $\bbeta=[\beta_0,\bbeta_1^T, \bbeta_2^T]$ and variance components $\btheta=[\sigma_{\alpha}^2,\sigma_e^2]^T$.% with true versions  $\dot\bbeta=[\dbeta_0,\dot\bbeta_1^T,\dot\bbeta_2^T]^T$  and $\dot\btheta=[\dsiga^2,\dsige^2]^T$, respectively.  
The REML criterion function is obtained by replacing the regression parameters $\bbeta$ in the log-likelihood (\ref{owlogli}) by their maximum likelihood estimators for each fixed $\btheta$ to produce a profile log-likelihood for $\btheta$ and then adding an adjustment term.  Let $\sz_i = [1, \sx_i^{(b)T}, \bar{\sx}_i^{(w)T}]^T$,  $\sw = [0, \bzero_{[1:p_b]}, \bs_w^{xyT}]^T$ and $\bw= \mbox{block diag}(0, \bzero_{[p_b:p_b]}, \bs_w^x)$. Then, for each fixed $\btheta$, we solve the estimating equations in (\ref{eqm}) for $\bbeta$ to obtain 
\begin{equation*}\label{profests}
\begin{split}
\hat{\bbeta}(\btheta) &= \bDelta(\btheta)^{-1}\Big(\sumig\tau_i\sz_i\yib + \sigma_e^{-2}\sw\Big), \quad \mbox{ with } \quad \bDelta(\btheta) = \sumig\tau_i\sz_i\sz_i^T +\sigma_e^{-2}\bw,
\end{split}
\end{equation*}
%where 
%\[
%\bDelta(\btheta) = \sumig\frac{{m_i}}{\sigma_e^2 + m_i\sigma_{\alpha}^2}\sz_i\sz_i^T + \frac{1}{\sigma_e^2}\bw .
%\]
%
and the REML criterion function is given by
\begin{equation*}\label{owloglir}
l_R(\btheta; \by)=  l(\hat{\bbeta}(\btheta), \btheta; \by) - \frac{1}{2} \log\left\{|\bDelta(\btheta)|\right\}.
\end{equation*}
The REML estimator $\hat{\btheta}_R$ of $\btheta$ is the maximiser of the REML criterion function $l_R(\btheta; \by)$; we call $\hat{\bbeta}_R=\hat{\bbeta}(\hat{\btheta}_R)$ the REML estimator of $\bbeta$ and write $\hbomega_R = (\hat{\beta}_{R0}, \hat{\bbeta}_{R1}^T, \hat{\sigma}_{R\alpha}^2, \hat{\bbeta}_{R2}^T, \hat{\sigma}_{Re}^2)^T$.  

Since $\bDelta(\btheta)$ does not depend on $\bbeta$, the REML estimator is also the maximiser of the adjusted log-likelihood
\begin{equation*} \label{adjlogli}
l_A(\bbeta, \btheta; \by)=  l(\bomega; \by) - \frac{1}{2} \log\left\{|\Delta(\btheta)|\right\}.
\end{equation*}
That is, we can find the REML estimator in one step instead of two \citep{patefield1977maximized} by maximising $l_A(\bbeta, \btheta; \by)$.  In either case, the estimating function is $\bpsi_{A}(\bomega)=[l_{A\beta_0}(\bomega), \,\ssl_{A\beta_1}(\bomega)^T, \,l_{A\siga^2}(\bomega), \,\ssl_{A\beta_2}(\bomega)^T, \,l_{A\sige^2}(\bomega)]^T$.   The derivatives $l_{A\beta_0}(\bomega)=l_{\beta_0}(\bomega)$, $ \ssl_{A\beta_1}(\bomega)=\ssl_{\beta_1}(\bomega)$ and $\ssl_{A\beta_2}(\bomega)=\ssl_{\beta_2}(\bomega)$, while
\begin{equation*}\label{extralikelhoo}
\begin{split}
&l_{A\siga^2}(\bomega) =l_{\siga^2}(\bomega) -\frac{1}{2}\mbox{trace}\Big\{\bDelta(\btheta)^{-1} \frac{\partial\bDelta(\btheta)}{\partial \sigma_{\alpha}^2}\Big\} \quad \mbox{ with } \quad \frac{\partial\bDelta(\btheta)}{\partial \sigma_{\alpha}^2} =  -\sumig\tau_i^2\sz_i\sz_i^T\\
&
l_{A\sige^2}(\bomega) =l_{\sige^2}(\bomega)-\frac{1}{2}\mbox{trace}\Big\{\bDelta(\btheta)^{-1}\frac{\partial\bDelta(\btheta)}{\partial \sigma_{e}^2}\Big\} \quad \mbox{ with } \\&\qquad \frac{\partial\bDelta(\btheta)}{\partial \sigma_{e}^2} = -\sumig m_i^{-1}\tau_i^2\sz_i\sz_i^T - \sigma_e^{-4}\bw.
\end{split}
\end{equation*}
%where
%\begin{displaymath}
%\begin{split}
%\frac{\partial\bDelta(\btheta)}{\partial \sigma_{\alpha}^2} &=  -\sumig\frac{m_i^2}{(\sigma_e^2 + m_i\sigma_{\alpha}^2)^{2}}\sz_i\sz_i^T,\qquad
%%
%\frac{\partial\bDelta(\btheta)}{\partial \sigma_{e}^2} = -\sumig\frac{m_i}{(\sigma_e^2 + m_i\sigma_{\alpha}^2)^2}\sz_i\sz_i^T - \frac{1}{\sigma_e^4}\bw.
%\end{split}
%\end{displaymath}
We show that the REML estimator is asymptotically equivalent to the maximum likelihood estimator by showing that the contribution from the adjustment terms to the estimating function is asymptotically negligible.  This yields the following theorem which we prove in Section \ref{sec:proof}.

\begin{theoremm}
	Suppose Condition A holds. Then, as $g, m_L \to\infty$, there is a solution $\hbomega_R$ to the adjusted likelihood estimating equations satisfying $|\bk^{1/2}(\hbomega_R-\dot\bomega)|=O_p(1)$ and
	\begin{displaymath}
	\bk^{\frac{1}{2}}(\hbomega_R-\hbomega)=\ o_p(1),
	\end{displaymath}
so Theorem 1 applies to the REML estimator.
\end{theoremm}

%%%%%%%%%%%%%%%%%%%%%%%%%%%%%%%%%%%%%%%%%%%%%%%%%%%%
\section{Discussion}\label{sec:disc}

Theorems 1 and 2 establish the asymptotic equivalence, asymptotic representations and asymptotic normality for the maximum likelihood and REML estimators of the parameters in the nested error regression model under very mild conditions when both the number of clusters and the cluster sizes increase to infinity.  In this section we interpret and discuss these results before pointing out possible directions for future work.

%\noindent 1.
We can estimate $\dot\bomega$ consistently when $g \to \infty$ with bounded cluster sizes but we need to let $m_L\to \infty$ to estimate the random effects $\{\alpha_i\}$ consistently \citep{jiang1998asymptotic}.  If $m_L\to \infty$ but $g$ is held fixed, we can estimate the within cluster variance $\dot\sige^2$ consistently but not the between cluster variance $\dot\siga^2$. These considerations motivate allowing both $g \to \infty$ and $m_L\to \infty$.  

%\noindent 2.
The asymptotic representation shows that the influence function of the maximum likelihood and REML estimators under the model is given by the summands of $\bb^{-1}\bxi$.  Explicitly, at a point $[ \alpha_i, e_{ij}, \sx^{(b)T}_i, \sx_{ij}^{(w)T}]^T$ (which we suppress in the notation), the influence function is the $(p_b+p_w+3)$-vector function $\blambda =[\lambda_{\beta_0}, \blambda_{\bbeta_1}^T, \lambda_{\siga^2}, \blambda_{\bbeta_2}^T, \lambda_{\sige^2}]^T$, where 
	\begin{equation*}
	\begin{split}
	&\lambda_{\beta_0}=\{(1- \ssc_1^T\bc_2^{-1}\sx_i^{(b)})/(1- \ssc_1^T\bc_2^{-1}\ssc_1 )\}\alpha_i,\\& \blambda_{\bbeta_1}=\{\bc_2^{-1}\sx_i^{(b)} + (1- \ssc_1^T\bc_2^{-1}\sx_i^{(b)})/(1- \ssc_1^T\bc_2^{-1}\ssc_1 )\}\sx_{i}^{(b)}\alpha_i,\\& % corrected
	\lambda_{\siga^2}= \alpha_i^2-\dsiga^2,\qquad
	\blambda_{\bbeta_2}=\bc_3^{-1}(\sx_{ij}^{(w)}-\bar{\sx}_{i}^{(w)})\eij \qquad
	\mbox{and} \qquad 
	\lambda_{\sige^2}=\eij^2-\dsige^2.
	\end{split}
	\end{equation*}
These expressions are not easy to obtain directly because the between and within parameters are estimated at different rates.   As is well-known, the estimators are not robust because the influence function is unbounded in the covariates, random effect and error.

%\noindent 3.
The central limit theorem allows us to construct asymptotic confidence intervals for the parameters in the model.  An asymptotic $100(1-\gamma)\%$ confidence interval for $\dot\beta_{1k}$ is 
\[
[\hat{\beta}_{1k} - \Phi^{-1}(1-\gamma/2)\hat{\sigma}_{\alpha}d_{kk}^{(b)1/2}/g^{1/2},\,\, \hat{\beta}_{1k} + \Phi^{-1}(1-\gamma/2)\hat{\sigma}_{\alpha}d_{kk(b)}^{1/2}/g^{1/2}],
\]
where $d_{kk}^{(b)}$ is the $k$th diagonal element of $\hat{\bc}_2^{-1} + \hat{\bc}_2^{-1}\hat{\ssc}_1\hat{\ssc}_1^T\hat{\bc}_2^{-1}/(1- \hat{\ssc}_1^T\hat{\bc}_2^{-1}\hat{\ssc}_1)$ with $\hat{\ssc}_1= g^{-1}\sum_{i=1}^g \sx_{i}^{(b)}$ and $\hat{\bc}_2 = g^{-1}\sum_{i=1}^g \sx_{i}^{(b)}\sx_{i}^{(b)^T}$, and an asymptotic
$100(1-\gamma)\%$ confidence interval for $\dot\beta_{2r}$ is 
\[
[\hat{\beta}_{2r} - \Phi^{-1}(1-\gamma/2)\hat{\sigma}_{e}d_{rr}^{(w)1/2}/n^{1/2},\,\, \hat{\beta}_{2r} + \Phi^{-1}(1-\gamma/2)\hat{\sigma}_{e}d_{rr}^{(w)1/2}/n^{1/2}],
\]
where $d_{rr}^{(w)}$ is the $r$th diagonal element of $(\bs_{w}^x/n)^{-1}$.  Setting the confidence interval on the $\log$ scale and then backtransforming, an asymptotic $100(1-\gamma)\%$ confidence interval for $\dot\sigma_{\alpha}$ is
\[
\begin{split}
[\hat{\sigma}_{\alpha}\exp\{- \Phi^{-1}(1-\gamma/2)&(\hat{\mu}_{4 \alpha}-\hat{\sigma}_{\alpha}^4)^{1/2} /2g^{1/2}\hat{\sigma}_{\alpha}^2\},\, \\&\hat{\sigma}_{\alpha}\exp\{ \Phi^{-1}(1-\gamma/2)(\hat{\mu}_{4 \alpha}-\hat{\sigma}_{\alpha}^4)^{1/2} /2g^{1/2}\hat{\sigma}_{\alpha}^2\}],
\end{split}
\]
where $\hat{\mu}_{4 \alpha} = g^{-1}\sum_{i=1}^g(\bar{y}-\hat{\beta}_0 - \sx_i^{(b)^T}\hat{\bbeta}_1 - \bar{\sx}_i^{(w)^T}\hat{\bbeta}_2)^4$ estimates $\E\alpha_1^4$.  Squaring the endpoints gives an asymptotic $100(1-\gamma)\%$ confidence interval for $\dot\sigma_{\alpha}^2$.

%\noindent 4.
The results show explicitly that the between and within parameters are estimated at different rates and the form of $\bc$ shows that, even without assuming normality, the maximum likelihood and REML estimators of the within parameters are asymptotically independent of the estimators of the between parameters. That is, the two sets of parameters are asymptotically orthogonal.  The within cluster regression parameter is asymptotically orthogonal  to the within cluster variance and the between cluster slope parameter is asymptotically orthogonal to the between cluster variance, but the intercept is only asymptotically orthogonal to the between cluster variance when the random effect distribution is symmetric.   

%\noindent 5.
When the cluster sizes are fixed, the maximum likelihood and REML estimators all converge to the true parameters at the same rate ($g^{-1/2}$) and the expression for their asymptotic variance is much more complicated.  Appending a subscript $m$ to emphasise that the cluster sizes are fixed at their upper bounds, the asymptotic variance of the estimators is $g^{-1}\bc_m = g^{-1}\bb_m^{-1}\ba_m\bb_m^{-1}$, where $\bb_m= - \lim_{g\to \infty} g^{-1}\E\nabla\psi(\dot\bomega)$ (which we can obtain from (\ref{bfin})) and $\ba_m= \lim_{g\to \infty} g^{-1}\var\{\psi(\dot\bomega)\}$.  We require assumptions on the convergence of weighted means and weighted products of covariates to ensure the existence of $\bb_m$ and $\ba_m$.  Under these assumptions, in general, $\bb_m$ is block diagonal for $[\bbeta^T, \btheta^T]^T$, although it is not block diagonal for $\bomega$ because the $(\siga^2,\sige^2)$ term is nonzero.  When $\bar{\sx}_i^{(w)} = \bzero_{[p_w:1]}$ for all $i = 1,\ldots, g$, the $(\bbeta_2, \beta_0)$ and $(\bbeta_2,\bbeta_1)$ terms are zero, but this does not affect the $(\siga^2,\sige^2)$.  The matrix $\ba_m$ involves third and fourth moments and is rarely evaluated in the non-normal case; general expressions are given in \cite{field2008bootstrap} and expressions specific to the model (\ref{mean}) are available from the authors on request.   It is in general not block diagonal for $[\bbeta^T, \btheta^T]^T$ unless both $\E(\alpha_1^3)=0$ and $\E(e_{11}^3)=0$.  The centering condition makes the covariance of $\ssl_{\bbeta_{2}}(\dot\bomega)$ with all the other components of $\bpsi(\dot\bomega)$ equal zero, but not the covariance between $l_{\beta_{0}}(\dot\bomega)$ or $\ssl_{\bbeta_{2}}(\dot\bomega)$ with $l_{\siga^2}(\dot\bomega)$ or $l_{\sige^2}(\dot\bomega)$.  These limits are not block diagonal for $\bomega$ even when both $\E(\alpha_1^3)=0$ and $\E(e_{11}^3)=0$ (because the covariance between $l_{\siga^2}(\dot\bomega)$ and $l_{\sige^2}(\dot\bomega)$ is nonzero).  Of course,  the nonzero terms in $\bc_m$ involve limits of weighted averages which differ from those in $\bc$.

%\noindent 6.
In working with the model (\ref{mean}), we have both between cluster and within cluster regression parameters to estimate.   If we have no between cluster covariates,we discard $\bbeta_1$, while if we have no within cluster covariates, we discard $\bbeta_2$.  The results for these cases can be obtained as special cases of the general results by deleting the components of vectors  and the rows and columns of matrices corresponding to the discarded parameter.  If there are no between cluster covariates in the model (there is no $\bbeta_1$ in the model), we drop rows and columns $2$ to $p_b+1$ from $\bc$.  If there are no within cluster covariates (there is no $\bbeta_2$ in the model), we drop rows and columns $p_b+3$ to $p_b+p_w+2$ from $\bc$.  There is a corresponding simplification to Condition A4.

%\noindent 7
We have treated the covariates in the model as fixed, conditioning on them when they are random.  As noted by \cite{yoon2020effect}, when the covariates are random, it makes sense to treat them as having a similar covariance structure to the response.  That is, $\sx_{i}^{(b)}$ are independent with mean $\bmu_x^{(b)}$ and variance $\bSigma_x^{(b)}$, and the $\sx_{ij}^{(w)}$ are independent in different clusters but correlated within clusters with mean $\bmu_x^{(w)}$, variance $\bUpsilon_x^{(w)}+\bSigma_x^{(w)}$ and within cluster covariance $\bUpsilon_x^{(w)}$.  The two types of covariates can be correlated.  Condition A holds if both covariates have finite $2+\delta$ moments.  We have $\ssc_1=\bmu_x^{(b)}$,$\bc_2=\bSigma_x^{(b)}+\bmu_x^{(b)}\bmu_x^{(b)T}$ and $\bc_3=\bSigma_x^{(w)}$.  If $\sx_{i}^{(b)}$ contains $\bar{\sx}_{i}^{(w)}$, the terms $\bmu_x^{(b)}$ and variance $\bSigma_x^{(b)}$ contain $\E \bar{\sx}_{i}^{(w)}=\bmu_x^{(w)}$, $\var(\bar{\sx}_{i}^{(w)}) =  \bUpsilon_x^{(w)}+m_i^{-1}\bSigma_x^{(w)} \to \bUpsilon_x^{(w)}$, as $m_L \to \infty$, and the covariance between $\sx_{i}^{(b)}$ and $\bar{\sx}_{i}^{(w)}$.

One motivation for allowing the cluster size to increase with the number of clusters is that, as we have noted, this is required for consistent prediction of the random effects \citep{jiang1998asymptotic}.  We have not considered prediction of the random effects explicitly in this paper but will do so in follow up work.  The present paper makes an important step towards tackling prediction for sample survey applications because our results allow subsampling within clusters.  In particular if the model (\ref{mean}) holds for the finite population, then noninformative subsampling of units within clusters ensures that the sample data satisfy the same model and hence that we can apply Theorems 1 and 2.

The model we have considered is a simple linear mixed model.  It is of interest to extend our results to more general linear mixed models and indeed to generalized linear mixed models.  It is clear that we can extend the hierarchical structure of the model and allow for more variance components.  The effect is to increase the sets of parameters so that there is a set for each level in the hierarchy.  The estimators in each level converge at different rates and the limit distribution has a diagonal block for each level in the hierarchy.   The maximum likelihood and REML estimators are not the only estimators of interest for the parameters of linear mixed models.  Other estimators (including robust estimators) are available and it is also of interest to derive their asymptotic properties.  We expect that  the form of the asymptotic covariance matrices for these estimators will be block diagonal with a separate block for the parameters at each level in the hierarchy, just as we found for the maximum likelihood and REML estimators.  Finally, \cite{jiang1996reml} also allowed the number of covariates to increase asymptotically and showed that the maximum likelihood and REML estimators have different asymptotic properties in this case.  This is also an interesting problem to consider in the framework of this paper.

%%%%%%%%%%%%%%%%%%%%%%%%%%%%%%%%%%%%%%%%%%%%%%%%%%%%

\section{Proofs}\label{sec:proof}

The proofs of Theorems 1 and 2 are presented in Subsection \ref{sec:part4}. The supporting lemmas used in these proofs are then proved in Subsections \ref{sec:part1} and \ref{sec:part2}.

\subsection {Proofs of Theorems 1 and 2}\label{sec:part4}

%\begin{proof}
\textit{Proof.}
	Write 
	\begin{displaymath}
	\begin{split}\label{expan}
	\bkmh\psi(\bomega)%&= \bkmh\bxi-\bb\bkh(\bomega-\dot\bomega) + \bkmh\{\psi(\dot\bomega)-\xi\}\\
	%&+\bkmh\E\{\psi(\bomega)-\psi(\dot\bomega)\} + \bb \bkh(\bomega-\dot\bomega)\\
	%&+ \bkmh[\psi(\bomega)-\psi(\dot\bomega)-\E\{\psi(\bomega)-\psi(\dot\bomega)\}].\\
	&= \bkmh\bxi-\bb\bkh(\bomega-\dot\bomega) + T_1 + T_2(\bomega) + T_3(\bomega), 
	\end{split}
	\end{displaymath}	
	where ${\bb} = \lim_{g\to\infty} \lim_{m_L\to\infty} -\bkmh\E\nabla\psi(\bomega)\bkmh$, $T_1  = \bkmh\{\psi(\dot\bomega)-\xi\}$, $T_2(\bomega) =\bkmh\E\{\psi(\bomega)-\psi(\dot\bomega)\} + \bb \bkh(\bomega-\dot\bomega)$ and $T_3(\bomega) = \bkmh[\psi(\bomega)-\psi(\dot\bomega)-\E\{\psi(\bomega)-\psi(\dot\bomega)\}]$.
%	\begin{displaymath}
%	\begin{split}
%	T_2(\bomega) &=\bkmh\E\{\psi(\bomega)-\psi(\dot\bomega)\} + \bb \bkh(\bomega-\dot\bomega), \\
%	\mbox{and } \quad T_3(\bomega) &= \bkmh[\psi(\bomega)-\psi(\dot\bomega)-\E\{\psi(\bomega)-\psi(\dot\bomega)\}].
%	\end{split}
%	\end{displaymath}	
%	%Let $\mathcal{N} = \{\bomega : \, |\bkh(\bomega-\dot\bomega)| \le M\}$, where $M$ is a  finite constant. 
If we can show that $|T_1| = o_p(1)$, $\underset{\bomega \in \mathcal{N}}{\sup} |T_2(\bomega)| = o(1)$ and $\underset{\bomega \in \mathcal{N}}{\sup} |T_3(\bomega)| = o_p(1)$, respectively, then uniformly on $\mathcal{N}$, we have
	\begin{equation}\label{lastone}
	\bkmh\psi(\bomega)=\bkmh\bxi- \bb \bkh(\bomega-\dot\bomega)+o_p(1).
	\end{equation}
	Multiplying by $(\bomega-\dot\bomega)^T\bkh$, 
	\begin{equation*}
	(\bomega-\dot\bomega)^T\psi(\bomega)=(\bomega-\dot\bomega)^T\bxi- (\bomega-\dot\bomega)^T\bkh\bb \bkh(\bomega-\dot\bomega)+o_p(1).
	\end{equation*}
	Since $B$ is positive definite, the right-hand side of is negative for $M$ sufficiently large. Therefore, according to Result 6.3.4 of \cite{ortega1970iterative}, a solution to the estimating equations exists in probability and satisfies $|\bkh(\hbomega-\dot\bomega)|=O_p(1)$, so $\hbomega \in \mathcal{N}$.  This allows us to substitute  $\hbomega$ for $\bomega$ in (\ref{lastone}) and rearrange the terms to obtain the asymptotic representation for $\hbomega$; the central limit theorem follows from the asymptotic representation and the central limit theorem for $\bxi$ that we establish in Lemma \ref{lem3}.
	
It remains to show that that remainder terms in  (\ref{expan}) are of smaller order and can be ignored.  In Lemma \ref{lem2}, we establish $|T_1| = o_p(1)$ by showing that the result holds for each component of $T_1= {\bk}^{-1/2}\{{\bpsi}(\dot\bomega)-\bxi\}$ by applying Chebychev's inequality and calculating the variances of the components. 
	
Our approach to handling $T_2(\bomega)$ and $T_3(\bomega)$ is inspired by \cite{bickel1975one} who applied similar arguments to one-step regression estimators.  The approach was extended to maximum likelihood and REML estimators in linear mixed models by \cite{richardson1994asymptotic}; the bounds we use require more care with increasing cluster size.  For  $T_2(\bomega)$, we have 
	\begin{displaymath}
	\begin{split}
	\underset{\bomega \in \mathcal{N}}{\sup} |T_2(\bomega)| & \le \underset{\bomega \in \mathcal{N}}{\sup} \left|\bkmh\left[\E\left\{ \bpsi(\bomega)-\bpsi(\dot\bomega)\right\} -\E\nabla\bpsi(\dot\bomega)(\bomega-\dot\bomega)\right] \right| \\& \quad+ \underset{\bomega \in \mathcal{N}}{\sup} |\bkmh\E\nabla\bpsi(\dot\bomega)(\bomega-\dot\bomega)+\bb\bkh(\bomega-\dot\bomega)| \\
		& \le  M\underset{\bomega \in \mathcal{N}}{\sup} \left\|\bk^{-1/2}
	\left\{\E\nabla\bpsi(\bOmega)-\E\nabla\bpsi(\dot\bomega)\right\}\bk^{-1/2}\right\| \\& \quad + M \|-\bkmh\E\nabla\bpsi(\dot\bomega)\bkmh-\bb\|\\
	& \le M\underset{\bomega \in \mathcal{N}}{\sup} \left\|\bk^{-1/2}
	\left\{\E\nabla\bpsi(\bOmega)-\E\nabla\bpsi(\dot\bomega)\right\}\bk^{-1/2}\right\| + M \|\bb_n-\bb\|,
	\end{split}
	\end{displaymath}
where the rows of $\bOmega$ are possibly different but lie between $\bomega$ and $\dot\bomega$ and $\bb_n=-\bkmh\E\nabla\psi(\dot\bomega)\bkmh$.  In Lemma \ref{lem9} we show that $\|\bb_n-\bb\|=o(1)$ and in Lemma \ref{lem5},
	\begin{equation*}\label{meqqp21}
	\underset{\bomega \in \mathcal{N}}{\sup} \left\|\bk^{-1/2}
	\left\{\E\nabla\bpsi(\bOmega)-\E\nabla\bpsi(\dot\bomega)\right\}\bk^{-1/2}\right\| =o(1).
	\end{equation*}

Finally, to handle $T_3(\bomega)$, decompose $\mathcal{N}=\{\bomega:|\bkh(\bomega-\dot\bomega)|\le M\}$ into  the set of $N=O(g^{\frac{1}{4}})$ smaller  cubes   $\mathcal{C}=\{\mathcal{C}(\st_k)\}$, where $\mathcal{C}(\st)= \{\bomega :|\bkh(\st-\dot\bomega)|\le Mg^{-1/4}\}$.  We first show that $|T_2(\bomega)| = o(1)$ holds over the  set of indices $\st_k=(t_{k1},t_{k2},t_{k3},t_{k4},t_{k5})^T$ for the cubes in $\mathcal{C}$ and then that the difference between taking the supremum over a fine grid  of points and over  $\mathcal{N}$ is small. Using Chebychev's inequality, for any $\eta>0$, we have 
	\begin{equation*}\label{meqpf1}
	\begin{split}
	\pr &\left( \max_{1\le k \le N}|\bkmh[\bpsi(\st_k)-\bpsi(\dot\bomega)-\E\{\bpsi(\st_k)-\bpsi(\dot\bomega)\}]|>\eta\right) \\&
	\le \sum_{k=1}^N\pr \left( |\bkmh[\bpsi(\st_k)-\bpsi(\dot\bomega)-\E\{\bpsi(\st_k)-\bpsi(\dot\bomega)\}]|>\eta\right) \\&
	\le\eta^{-2}\sum_{k=1}^N\E|\bkmh[\bpsi(\st_k)-\bpsi(\dot\bomega)-\E\{\bpsi(\st_k)-\bpsi(\dot\bomega)\}]|^2\\&
	= \eta^{-2}g^{-1}\sum_{k=1}^N\E |\bpsibb(\st_k)-\bpsibb(\dot\bomega)-\E\{\bpsibb(\st_k)-\bpsibb(\dot\bomega)\}|^2 \\&\qquad+\eta^{-2}n^{-1}\sum_{k=1}^N\E
	|\bpsiww(\st_k)-\bpsiww(\dot\bomega)-\E\{\bpsiww(\st_k)-\bpsiww(\dot\bomega)\}|^2\\
	& =  \eta^{-2}g^{-1}\sum_{k=1}^N\mbox{trace}[\var\{\bpsibb(\st_k)-\bpsibb(\dot\bomega)\}]\\&\qquad
	+\eta^{-2}n^{-1}\sum_{k=1}^N\mbox{trace}[\var\{\bpsiww(\st_k)-\bpsiww(\dot\bomega)\}].
\end{split}
\end{equation*}
We show in Lemma \ref{lem7} that the variances $\var\{\psibb(\st_k)-\psibb(\dot\bomega)\}$ and $\var\{\psiww(\st_k)-\psiww(\dot\bomega)\}$ are uniformly bounded by $L$, say, so
\begin{equation*}
\begin{split}
\pr \Big( \max_{1\le k \le N} |\bkmh[\bpsi(\st_k)-&\bpsi(\dot\bomega)-\E\{\bpsi(\st_k)-\bpsi(\dot\bomega)\}]|>\eta\Big)\\&
 \le  \eta^{-2}L N \{g^{-1}(2+p_b) + n^{-1}(1+p_w)\}= o(1),
 \end{split}
\end{equation*}
using the fact that $N= O(g^{1/4})$. 
Using Taylor expansion, we get 
	\begin{equation*}\label{meeq222}
	\begin{split}
	\underset{1\le k\le N}{\max}\underset{\bomega\in\mathcal{C}(\st_k)}{\sup}&|\bkmh[\bpsi(\bomega)-\bpsi(\st_k)-\E\{\bpsi(\bomega)-\bpsi(\st_k)\}]|\\&
	= \underset{1\le k\le N}{\max}\underset{\bomega\in\mathcal{C}(\st_k)}{\sup}|\bkmh[\nabla\bpsi(\bOmega_k)(\bomega - \st_k)-\E\nabla\bpsi(\bOmega_k)(\bomega-\st_k)]|
	\\&
	\le M \underset{\bomega \in\mathcal{N}}{\sup} g^{-1/4}\|\bkmh\{\nabla\bpsi(\bOmega)-\E\nabla\bpsi(\bOmega)\}\bkmh\|,
	%\\&
	%\le M \underset{1\le k\le N}{\max}\underset{\bomega\in\mathcal{C}(\st_k)}{\sup} g^{-5/4}\|\nabla\bpsi^{(bb)}\{\bOmega^{(b)}_{(s)}\}-\E\nabla\bpsi^{(bb)}\{\bOmega^{(b)}_{(s)}\}\}\|
	%\\&
	%+ M \underset{1\le k\le N}{\max}\underset{\bomega\in\mathcal{C}(\st_k)}{\sup} g^{-3/4}n^{-1/2}\|\nabla\bpsi^{(bw)}\{\bOmega^{(b)}_{(s)}\}-\E\nabla\bpsi^{(bw)}\{\bOmega^{(b)}_{(s)}\}\}\| \\&
	%+ M \underset{1\le k\le N}{\max}\underset{\bomega\in\mathcal{C}(\st_k)}{\sup} g^{-3/4}n^{-1/2}\|\nabla\bpsi^{(wb)}\{\bOmega^{(w)}_{(s)}\}-\E\nabla\bpsi^{(wb)}\{\bOmega^{(w)}_{(s)}\}\}\| \\&
	%+ M \underset{1\le k\le N}{\max}\underset{\bomega\in\mathcal{C}(\st_k)}{\sup} g^{-1/4}n^{-1}\|\nabla\bpsi^{(ww)}\{\bOmega^{(w)}_{(s)}\}-\E\nabla\bpsi^{(ww)}\{\bOmega^{(w)}_{(s)}\}\}\| 
	\end{split}
	\end{equation*}	
	where the rows of $\bOmega_k$ are between $\st_k$ and $\bomega$. The result follows from Lemma \ref{lem8}.  %\end{proof}
\hfill $\Box$ \\

	%\section{Maximum Adjusted Likelihood Estimator}

%\begin{lma} \label{lem10}
%	Suppose Condition A holds. Then, as $g, m_{L} \to\infty$, there is a solution $\hbtheta_A$ to the maximum adjusted likelihood estimating equation $\bzero = \bpsi_{A}(\btheta)$, satisfying $|\bk^{\frac{1}{2}}(\hbtheta_A-\dot\btheta)|=O_p(1)$. Moreover,
%	\begin{displaymath}
%	\bk^{\frac{1}{2}}(\hbomega_A-\hbomega)=\ o_p(1),
%	\end{displaymath}
%	so the maximum adjusted likelihood estimator has the same asymptotic distribution as the maximum likelihood estimator.
%\end{lma}
%\begin{proof}
\textit{Proof.}
	Let $\bk_{\bbeta} = \mbox{diag}(g, g\bone_{p_b}, n\bone_{p_w})$ and write
	\begin{equation*}\label{extralikelhoo1}
	\begin{split}
	|g^{-1/2}\{ l_{A\siga^2}(\bomega) - l_{\siga^2}(\bomega)\} |&\le %\frac{1}{2g^{1/2}}|\mbox{trace}\Big\{\bDelta(\btheta)^{-1} \frac{\partial\bDelta(\btheta)}{\partial \sigma_{\alpha}^2}\Big\}| = 
	\frac{1}{2g^{1/2}}|\mbox{trace}\Big\{\bk_{\bbeta}^{1/2}\bDelta(\btheta)^{-1} \bk_{\bbeta}^{1/2}\bk_{\bbeta}^{-1/2}\frac{\partial\bDelta(\btheta)}{\partial \sigma_{\alpha}^2}\bk_{\bbeta}^{-1/2}\Big\}|\\
	|n^{-1/2}\{l_{A\sige^2}(\bomega) - l_{\sige^2}(\bomega)\}| &\le   \frac{1}{2n^{1/2}}|\mbox{trace}\Big\{\bk_{\bbeta}^{1/2}\bDelta(\btheta)^{-1} \bk_{\bbeta}^{1/2}\bk_{\bbeta}^{-1/2}\frac{\partial\bDelta(\btheta)}{\partial \sigma_{e}^2}\bk_{\bbeta}^{-1/2}\Big\}|.
	\end{split}
	\end{equation*}
Then from Lemma 4 and the arguments establishing the convergence of $\bb_n$ to $\bb$, we can show that uniformly in $\bomega\in \mathcal{N}$ as $g, m_L \rightarrow \infty$, the matrices $\bk_{\bbeta}^{-1/2}\bDelta(\btheta)\bk_{\bbeta}^{-1/2}$, $\bk_{\bbeta}^{-1/2}\frac{\partial\bDelta(\btheta)}{\partial \sigma_{\alpha}^2}\bk_{\bbeta}^{-1/2}$ and $\bk_{\bbeta}^{-1/2}\frac{\partial\bDelta(\btheta)}{\partial \sigma_{e}^2}\bk_{\bbeta}^{-1/2}$ all converge to $(p_b+p_w+1)\times (p_b+p_w+1)$ matrices with finite elements.  Consequently, both $|g^{-1/2}\{ l_{A\siga^2}(\bomega) - l_{\siga^2}(\bomega)\} |=o_p(1)$ and $|n^{-1/2}\{l_{A\sige^2}(\bomega) - l_{\sige^2}(\bomega)\}|=o_p(1)$ uniformly in $\bomega\in \mathcal{N}$, and the result follows from Theorem 1.  %\end{proof}
\hfill $\Box$ \\

%%%%%%%%%%%%%%%%%%%%%%%%%%%%%%%%%%%%%%%%%%%%%%%%%%%%%

\subsection{Lemmas for the estimating function $\bpsi$} \label{sec:part1}

We prove a central limit theorem for $\bxi$ and that $T_1 = o_p(1)$ (i.e $\bpsi(\dot\bomega)$ can be approximated by $\bxi$).  We also prove that the variances of the components of $\bpsi(\bomega)-\bpsi(\dot\bomega)$ are uniformly bounded

%\subsection{The Central Limit Theorem for $\bxi$}

\begin{lma}  \label{lem3}
	Suppose  Condition A holds.
	Then, as $g, m_L \to\infty$, 
	$\bkmh\bxi \xrightarrow{D} N(\bzero,\ba)$, where
\begin{equation}\label{ggg}
\begin{split}
&\ba=\\&
\left[ \begin{matrix}
1/\dsiga^2&\ssc_1^T/\dsiga^2& \E\alpha_1^3/(2\dsiga^6)&\bzero_{[1:p_w]}&0\\
\ssc_1/\dsiga^2 & \bc_2/\dsiga^2& \ssc_1\E\alpha_1^3/(2\dsiga^6)&\bzero_{[p_b:p_w]}&\bzero_{[p_b:1]}\\
\E\alpha_1^3/(2\dsiga^6) & \ssc_1^T\E\alpha_1^3/(2\dsiga^6)& (\E\alpha_1^4-\dsiga^4)/(4\dsiga^8)&\bzero_{[1:p_w]}&0\\
\bzero_{[p_w:1]} &\bzero_{[p_w:p_b]}&\bzero_{[p_w:1]}& \bc_{3}/\dsige^2&\bzero_{[p_w:1]}\\
0&\bzero_{[1:p_b]}&0&\bzero_{[1:p_w]}& (\E e_{11}^4-\dsige^4)/(4\dsige^8)
\end{matrix}\right].
\end{split}
\end{equation}
\end{lma}
\begin{proof} 
The components of $\bxi$ are sums of independent random variables with zero means and finite variances.
Since $\{\alpha_i\}$ and $ \{e_{ij}\}$ are independent,  it is straightforward to compute $\var\{\bkmh \bxi\}=\ba_{n}$, %where
%\begin{displaymath}
%\begin{split}
%&\ba_{n} %=\var\left[ \begin{matrix}
%%\gmh\bxibb \\n^{-\frac{1}{2}}\bxiww
%%\end{matrix}\right] 
%%=\left[ \begin{matrix}
%%\var[\gmh\bxibb]&\bzero\\
%%\bzero&\var[n^{-\frac{1}{2}}\bxiww]
%%\end{matrix}\right] 
%%\\&
%=\\&\left[ \begin{matrix}
%1/\dot\sigma_{\alpha}^2&g^{-1}\sumig \sx_i^{(b)T}/\dsiga^2&\E\alpha_1^3/(2\dsiga^6)&\bzero_{[1:p_w]}&0\\
%g^{-1}\sumig \sx_i^{(b)}/\dsiga^2& g^{-1}\sumig \sx_i^{(b)}\sx_i^{(b)T}/\dsiga^2&
%g^{-1}\sumig \sx_i^{(b)} \E\alpha_1^3/(2\dsiga^6)&\bzero_{[p_b:p_w]}&\bzero_{[p_b:1]}\\
%\E\alpha_1^3/(2\dsiga^6)& g^{-1}\sumig \sx_i^{(b)T}\E\alpha_1^3/(2\dsiga^6)&(\E\alpha_1^4-\dsiga^4)/(4\dsiga^8)&\bzero_{[1:p_w]}&0\\
%\bzero_{[p_w:1]}&\bzero_{[p_w:p_b]}&\bzero_{[p_w:1]}&n^{-1}\bs_{w}^x/\dsige^2&\bzero_{[p_w:1]}\\
%0&\bzero_{[1:p_b]}&0 &\bzero_{[1:p_w]}&(\E e_{11}^4-\dsige^4)/(4\dsige^8)
%\end{matrix}\right].
%\end{split}
%\end{displaymath}
and then from Condition A4, as $g, \, m_L \rightarrow \infty$, to show that $\ba_{n} \rightarrow \ba$.  Partition $\bxi$ into $\bxibb$ containing the first $p_b+2$ elements (corresponding to the between parameters) and $\bxiww$ containing the remaining $p_w+1$ elements (corresponding to the within parameters). Partition $\ba$ conformably into the block diagonal matrix with diagonal blocks $\ba_{11}$ and $\ba_{22}$, where  $\ba_{11}$ is $(p_b+2) \times (p_b+2)$ and $\ba_{22}$ is $(p_w+1) \times (p_w+1)$.  We prove that $g^{-1/2}\bxi^{(b)}\xrightarrow{D} N(\bzero,\ba_{11})$ and $n^{-1/2}\bxi^{(w)}\xrightarrow{D} N(\bzero,\ba_{22})$, and  the result then follows from the fact that $\bxi^{(b)}$ and $\bxi^{(w)}$ are independent.
	
	Write $\bxi^{(b)} = \sum_{i=1}^g \bxi_{i}^{(b)}$, where the summands $\bxi_{i}^{(b)} = [\xi_{i\beta_{0}}, \bxi_{i\bbeta_{1}}^T, \xi_{i\siga^2}]^T$, and let $\sa$ be a  fixed $(p_b+2)$-vector satisfying $\sa^T\sa=1$. Then $g^{-1/2}\sa^T\bxi^{(b)}$ is a sum of independent scalar random variables with mean zero and finite variance.  It follows from the $c_r$-inequality and Conditions A3 - A4 that Lyapunov's condition holds.
%As $g \to\infty$,  $\var\{g^{-1/2}\sa^T\bxi^{(b)}\}\to \sa^T\ba_{11}\sa$.  Also, from conditions A3 and A4, there exists a $\delta>0$ such that 
%	\begin{displaymath}
%	\begin{split}
%	\sumig\E |g^{-1/2}\sa^T\bxi_{i}^{(b)}|^{2+\delta}  
%	& \le  g^{-1-\delta/2}
%	\sumig\E|\bxi_{i}^{(b)}|^{2+ \delta}\\&
%	%	=  g^{-1-\frac{\delta}{2}}
%	%	\sumig\E[\{\xi_{i\beta_{0(s)}}^2+\xi_{i\betai}^2+\xi_{i\sigass^2}^2\}^{1+\frac{\delta}{2}}]\\&
%	\le  g^{-1-\delta/2} 4^{1+\delta}
%	\sumig \{
%	\E|\xi_{i\beta_{0}}|^{2+\delta}+\E|\bxi_{i\bbeta_{1}}|^{2+\delta}+\E|\xi_{i\siga^2}|^{2+\delta}\}
%	\end{split}
%	\end{displaymath}
%	by the $c_r$-ineqality.   From (\ref{mmeq1}),  $\E|\xi_{i\beta_{0}}|^{2+\delta}=\dsiga^{-2(2+\delta)}\E|\alpha_1|^{2+\delta}$, 
%	$\E|\bxi_{i\bbeta_{1}}|^{2+\delta}=\dsiga^{-2(2+\delta)}|\sx_i^{(b)}|^{2+\delta}\E|\alpha_1|^{2+\delta}$ and 
%	$\E|\xi_{i\siga^2}|^{2+\delta}=(2\dsiga^{4})^{-(2+\delta)}\E|\alpha_1^2-\dsiga^2|^{2+\delta}$, so by 
%	Conditions A3 and A4
%	\begin{displaymath}
%	(\sa^T\ba_{11}\sa)^{-2-\delta}
%	\sumig\E |g^{-\frac{1}{2}}\sa^T\bxi_{i}^{(b)}|^{2+\delta} =O_p(g^{-\delta/2})=o_p(1)
%	\end{displaymath}
%	and Lyapunov's condition holds. 
Consequently $g^{-1/2}\sa^T\bxi^{(b)}$ converges in distribution to $N(0,\sa^T\ba_{11}\sa)$, as $g\to\infty$ and the result follows from  the Cramer-Wold device  \citep[p 49]{billingsley1968convergence}.
	The proof that $n^{-\frac{1}{2}}\bxi^{(w)}$ converges to $N (0,\ba_{22})$, as $g, m_L \to\infty$, is similar.
\end{proof}

In the proofs of Lemmas 2-6, we use the following simple bounds which we gather here for convenience.  Uniformly on $\bomega \in \mathcal{N}$, there exist fixed constants $0 < L_1 < L_2 < \infty$ such that both $m_i L_2^{-1} \le m_i\dot\sigma_{\alpha}^2 \le \dot\sigma_{e}^2+ m_i\dot\sigma_{\alpha}^2 =  m_i\dot\sigma_{\alpha}^2\{\dot\sigma_{e}^2/m_i\dot\sigma_{\alpha}^2  + 1\} \le m_i L_1^{-1}$ and  for $g$ sufficiently large, $m_iL_2^{-1} \le \sigma_{e}^2+ m_i\sigma_{\alpha}^2 \le m_i L_1^{-1}$ hold.  %and $|\dot\sigma_{e}^2/m_i + \dot\sigma_{\alpha}^2 + \sigma_{e}^2/m_i + \sigma_{\alpha}^2| \le L_3$, where $L_3 < \infty$ are fixed constants.and $m_L$
 It follows that uniformly both in $\bomega \in \mathcal{N}$ and $1\le i \le g$, %$L_1 \le \tau_i, \dot\tau_i \le L_2$, and then, for $m_L$ sufficiently large
\begin{equation} \label{bounds}
\begin{split}
& L_1 \le \tau_i, \dot\tau_i \le L_2, \qquad |\tau_i - \dot\tau_i| \le L_2^2M(m_L^{-1}n^{-1/2} +g^{-1/2}) \le O(g^{-1/2})\\
&|\tau_i^2(1- \dot\tau_i^{-1}\tau_i)| = O(g^{-1/2}), \quad
|\tau_i^2 - \dot\tau_i^2| \le |\tau_i - \dot\tau_i||\tau_i+ \dot\tau_i|  = O(g^{-1/2}).
\end{split}	
\end{equation}
We also require the moments of $\bar{e}_i = m_i^{-1}\sum_{j=1}^{m_i}e_{ij}$ which are  
\begin{equation} \label{moments}
\begin{split}
& \E\bar{e}_i=0, \qquad\var \bar{e}_i=m_i^{-1}\dot\sigma_e^2, \qquad \E(\bar{e}_{i}^3) =  m_i^{-2}\E e_{11}^3, \\&
\E(\bar{e}_{i}^4) =  m_i^{-2}3\dot\sigma_e^4 + m_i^{-3}\{\E e_{11}^4-3\dot\sigma_e^4\} \le m_i^{-2}3\dot\sigma_e^4 + m_i^{-3}\E e_{11}^4;
\end{split}	
\end{equation}
 see for example \cite[p 345]{cramer1946mathematical}.  These imply that $\var(\alpha_i+\bar{e}_i) = (\dot\sigma_{\alpha}^2 + m_i^{-1}\dot\sigma_e^2) = \dot\tau_i^{-1}$.

\begin{lma}  \label{lem2}
Suppose Condition A holds.  Then $|T_1|=o_p(1)$. 
\end{lma}

\begin{proof}
We establish the result for each component of $T_1= {\bk}^{-1/2}\{{\bpsi}(\dot\bomega)-\bxi\}$.
	Write $\bar{y}_i=\dbeta_{0}+\sx_i^{(b)T}\dot{\bbeta}_{1}+\bar{\sx}_{i}^{(w)T}\dot{\bbeta}_{2}+\alpha_i+\bar{e}_{i} = \sz_i^T\dot\bbeta + \alpha_i+\bar{e}_{i}$, where $\bar{e}_i = m_i^{-1}\sum_{j=1}^{m_i}e_{ij}$, and $\yij-\bar{y}_i =(\sx_{ij}^{(w)}-\bar{\sx}_{i}^{(w)})^T\dot{\bbeta}_{2}+\eij-\bar{e}_{i}$.  
	Then
	\begin{equation*}
	\begin{split}
	\ssl_{\bbeta_{1}}(\dot\bomega)- \bxi_{\bbeta_{1}}%&=\sumig\frac{{m_i}\sx_{i}^{(b)}}{\dot\sigma_{e}^2 + m_i\dot\sigma_{\alpha}^2}(\yib-\dot\beta_{0}-\sx_{i}^{(b)T}\dot\bbeta_{1}-\sxbisw\dot\bbeta_{2})- \frac{1}{\dot\sigma_{\alpha}^2}\sumig\sx_i^{(b)}\alpha_i\\
%	&= \sumig\frac{{m_i}\sx_i^{(b)}}{\dot{\sigma}_{e}^2 + m_i\dot{\sigma}_{\alpha}^2} (\alpha_i+\bar{e}_{i})-\frac{1}{\dot\sigma_{\alpha}^2}\sumig\sx_i^{(b)}\alpha_i\\
	&= \sumig(\dot\tau_i-1/\dot\sigma_{\alpha}^2)\sx_i^{(b)}\alpha_i +\sumig\dot\tau_i\sx_i^{(b)}\bar{e}_{i}. 
	\end{split}
	\end{equation*}
	The $k$th components of the two sums in the last line have mean zero and variances
	\begin{equation*}
	\var \{ \sumig ( \dot\tau_i-1/\dot\sigma_{\alpha}^2) x_{ik}^{(b)}\alpha_i\}=\sumig m_i^{-2} \dot\tau_i^2 \dot\sigma_e^4x_{ik}^{(b)2}\dot\sigma_{\alpha}^2 % \le m_L^{-2} L^2 \dot\sigma_e^4\dot\sigma_{\alpha}^2\sumig x_{ik}^{(b)2}  
	=O(m_L^{-2}g)
	\end{equation*}
	and
	\begin{equation*}
	\var \{  \sumig\dot\tau_i x_{ik}^{(b)}\bar{e}_{i(s)} \} = \sumig m_i^{-1}\dot\tau_i^2 \dot\sigma_e^2x_{ik}^{(b)2} % \le m_L^{-1} L^2 \dot\sigma_e^2 \sumig x_{ik}^{(b)2} 
	=O(m_L^{-1}g),
	\end{equation*}
	respectively, using (\ref{bounds}) and (\ref{moments}).  It follows that $\ssl_{\bbeta_{1}}(\dot\bomega)- \bxi_{\bbeta_{1}} = o_p(g^{1/2})$
	and, by essentially the same argument, $l_{\beta_{0}}(\dot\bomega)- \xi_{\beta_{0}} = o_p(g^{1/2})$.
	For the estimating equation for the between variance component, write
	\begin{equation*}
	\begin{split}
	l_{\sigma_{\alpha}^2}(\dot\bomega)- \xi_{\sigma_{\alpha}^2} %\\&
%	=-\frac{1}{2}\sumig\frac{{m_i}}{\dot\sigma_{e}^2 + m_i\dot\sigma_{\alpha}^2}+\frac{1}{2}\sumig\frac{{m_i}^2}{(\dot{\sigma}_{e}^2 + m_i\dot{\sigma}_{\alpha}^2)^2}(\yib-\dot\beta_{0}-\sx_{i}^{(b)T}\dot\bbeta_{1}-\sxbisw\dot\bbeta_{2})^2 
%	-\frac{1}{2\dot\sigma_{\alpha}^4} \sumig(\alpha_i^2-\dot\sigma_{\alpha}^2)\\
%		&=-\frac{1}{2}\sumig\frac{{m_i}}{\dot\sigma_{e}^2 + m_i\dot\sigma_{\alpha}^2}+\frac{1}{2}\sumig\frac{{m_i}^2}{(\dot{\sigma}_{e}^2 + m_i\dot{\sigma}_{\alpha}^2)^2}(\alpha_i^2+2\bar{e}_{i}\alpha_i+\bar{e}_{i}^2)
%	-\frac{1}{2\dot\sigma_{\alpha}^4} \sumig(\alpha_i^2-\dot\sigma_{\alpha}^2)\\&
	= \frac{1}{2}\sumig(\dot\tau_i^2 - 1/\dot\sigma_{\alpha}^4)(\alpha_i^2-\dot\sigma_{\alpha}^2) +\frac{1}{2}\sumig\dot\tau_i^2 (2\bar{e}_{i}\alpha_i+\bar{e}_{i}^2 - \dot{\sigma}_{e}^2/m_i).
	\end{split}
	\end{equation*}
From (\ref{bounds}) and (\ref{moments}), the variances of the sums are $O(m_L^{-2}g)$ and $O(m_L^{-1}g)$ so $l_{\siga^2}(\dot\bomega) - \xi_{\siga^2} = o_p(g^{1/2})$.
%	Twice the first sum has mean zero and
%	\begin{displaymath}
%	\begin{split}
%	\var  \left[\sumig\left\{\frac{{m_i}^2}{(\dot{\sigma}_{e}^2 + m_i\dot{\sigma}_{\alpha}^2)^2} - \frac{1}{\dot\sigma_{\alpha}^4}\right\}(\alpha_i^2-\dot\sigma_{\alpha}^2) \right]=\sumig \frac{ (\dot\sigma_e^4+2 m_i\dot\sigma_{\alpha}^2\dot\sigma_e^2)^2     }{\dot\sigma_{\alpha}^8(\dot{\sigma}_{e}^2 + m_i\dot{\sigma}_{\alpha}^2)^4}(\E\alpha_1^4-\dot\sigma_{\alpha}^4)
%	=O(m_L^{-2}g),
%	\end{split}
%	\end{displaymath}
%	and, using the fact that  $\E(\bar{e}_{i}^4) =  m_i^{-2}3\dot\sigma_e^4 + m_i^{-3}\{\E e_{11}^4-3\dot\sigma_e^4\} \le m_i^{-2}3\dot\sigma_e^4 + m_i^{-3}\E e_{11}^4$ \citep[p 345]{cramer1946mathematical},  twice the second sum has mean zero and
%	\begin{align*}
%	\var  \Big\{\sumig\frac{m_i^2}{(\dot{\sigma}_{e}^2 + m_i\dot{\sigma}_{\alpha}^2)^2} & (2\bar{e}_{i}\alpha_i+\bar{e}_{i}^2 - \dot{\sigma}_{e}^2/m_i)\Big\}
%%	&=\sumig \frac{ m_i^4}{(\dot{\sigma}_{e}^2 + m_i\dot{\sigma}_{\alpha}^2)^4}\E\{4\bar{e}_{i}^2\alpha_i^2+(\bar{e}_{i}^2 - \dot{\sigma}_{e}^2/m_i)^2 + 4\bar{e}_{i}\alpha_i(\bar{e}_{i}^2 - \dot{\sigma}_{e}^2/m_i)\}\\
%	\le \sumig \frac{ m_i^4}{(\dot{\sigma}_{e}^2 + m_i\dot{\sigma}_{\alpha}^2)^4}\{m_i^{-1}4\dot{\sigma}_{e}^2\dot{\sigma}_{\alpha}^2 +m_i^{-2}2 \dot\sigma_e^4 +m_i^{-3}\E e_{11}^4)=O(m_L^{-1}g),
%	\end{align*}
%	so $l_{\siga^2}(\dot\bomega) - \xi_{\siga^2} = O_p(m_L^{-1/2}g^{1/2})$.
	
	Next, we can write $\bs_{w}^{xy} = \bs_{w}^x\dot{\bbeta}_{2}+\sumig\sum_{j=1}^{m_i}(\sx_{ij}^{(w)}-\bar{\sx}_{i}^{(w)})e_{ij}$
%	\begin{align}\label{swxy}
%	\bs_{w}^{xy} &= \bs_{w}^x\dot{\bbeta}_{2}+\sumig\sum_{j=1}^{m_i}(\sx_{ij}^{(w)}-\bar{\sx}_{i}^{(w)})e_{ij}
%	\end{align}
	so, for the within slope parameter, 
	\begin{equation*}
	\begin{split}
\ssl_{\bbeta_{2}}(\dot\bomega) - \bxi_{\bbeta_{2}} %&=
%	\frac{1}{\dot\sigma_e^2}\bs_{w}^{xy}-\frac{1}{\dot\sigma_e^2}\bs_{w}^x\dot\bbeta_{2}+\sumig\frac{m_i\bar\sx_{i}^{(w)}}{{\dot\sigma}_{e}^2 + m_i{\dot\sigma}_{\alpha}^2}(\yibs-\dot\beta_{0}-\sx_{i}^{(b)T}\dot\bbeta_{1}-\sxbisw\dot\bbeta_{2})  - \frac{1}{\dot\sigma_e^2}\sumig\sum_{j=1}^{m_i}(\sx_{ij}^{(w)}-\bar{\sx}_{i}^{(w)})e_{ij}\\
&=
	\sumig \dot\tau_i \bar\sx_{i}^{(w)}(\alpha_i + \bar{e}_{i}) =o_p(n^{-1/2}),
	\end{split}
	\end{equation*}
because $g=o(n)$.
%Since the $k$th component of the difference $\ssl_{\bbeta_{2}}(\dot\bomega) - \bxi_{\bbeta_{2}}$ has variance
%	\begin{equation*}
%	\begin{split}
%\var\left\{
%	\sumig\frac{m_i\bar\sx_{ik}^{(w)}}{{\dot\sigma}_{e}^2 + m_i{\dot\sigma}_{\alpha}^2}(\alpha_i + \bar{e}_{i})\right\}&=\sumig\frac{m_i^2\bar\sx_{ik}^{(w)2}}{({\dot\sigma}_{e}^2 + m_i{\dot\sigma}_{\alpha}^2)^2}(\dot\sigma_{\alpha}^2 + \dot\sigma_e^2/m_i) =O(g) +O(g/m_L)
%	\end{split}
%	\end{equation*}
%and $g=o(n)$, we have $\ssl_{\bbeta_{2}}(\dot\bomega) - \bxi_{\bbeta_{2}} = o_p(n^{1/2})$.
	For the within variance component, expanding $S_w^y$, we show that
	\begin{equation}\label{bbbb}
	\begin{split}
	&\swy-2\dot{\bbeta}_2^T\bs_w^{xy}+\dot{\bbeta}_2^T\bs_w^x\dot{\bbeta}_2=(n-g)\dot\sigma_e^2+\sumig\sum_{j=1}^{m_i}(e_{ij}^2-\dot\sigma_e^2)+o_p(n^{1/2}).
	\end{split}
	\end{equation}
	It then follows that
	\begin{equation*}
	\begin{split}
	l_{\sigma_e^2}(\dot\bomega)-\xi_{\sigma_e^2}
%	&=-\frac{1}{2}\sumig\frac{1}{\dot\sigma_{e}^2 + m_i\dot\sigma_{\alpha}^2}-\frac{n-g}{2\dot\sigma_e^2}+\frac{1}{2\dot\sigma_e^4}(S_{w}^y-2\bs_{w}^{xyT}\dot\bbeta_{2}+\dot\bbeta_{2}^T\bs_{w}^x\dot\bbeta_{2})\\&\qquad+\frac{1}{2}\sumig\frac{{m_i}}{(\dot\sigma_{e}^2 + m_i\dot\sigma_{\alpha}^2)^2}(\yib-\dot\beta_{0}-\sx_{i}^{(b)T}\dot\bbeta_{1}-\sxbisw\dot\bbeta_{2})^2   - \frac{1}{2\dot\sigma_e^4}\sumig\sum_{j=1}^{m_i}(e_{ij}^2-\dot\sigma_e^2)\\
%	& = -\frac{1}{2}\sumig\frac{1}{\dot\sigma_{e}^2 + m_i\dot\sigma_{\alpha}^2} +\frac{1}{2}\sumig\frac{{m_i}}{(\dot\sigma_{e}^2 + m_i\dot\sigma_{\alpha}^2)^2}(\alpha_i^2 + 2\alpha_i\bar{e}_{i} + \bar{e}_{i}^2)+ o_p(n^{1/2})\\
	& =\frac{1}{2}\sumig m_i^{-1}\dot\tau_i^2(\alpha_i^2 - \dot\sigma_{\alpha}^2+ 2\alpha_i\bar{e}_{i} + \bar{e}_{i} ^2 - \dot\sigma_e^2/m_i)+ o_p(n^{1/2}).
	\end{split}
	\end{equation*}
	Since $\E(\alpha_i^2 - \dot\sigma_{\alpha}^2+ 2\alpha_i\bar{e}_{i} + \bar{e}_{i}^2 - \dot\sigma_e^2/m_i)^2 = \E(\alpha_1^2 - \dot\sigma_{\alpha}^2)^2 + m_i^{-1}4\dot\sigma_{\alpha}^2\dot\sigma_{e}^2 + \var(\bar{e}_{i}^2)$, we have $l_{\sigma_e^2}(\dot\bomega) - \xi_{\sigma_e^2} =o_p(n^{1/2})$,
	which completes the proof.  \end{proof}

\begin{lma}  \label{lem7}
	Suppose Condition A holds. Then,  there exists a finite constant $L$ such that
	\begin{equation*}
	\begin{split}
	& \underset{\bomega\in \mathcal{N}}{\sup}\var\{l_{\beta_{0}}(\bomega)-l_{\beta_{0}}(\dot\bomega)\} \le L, \quad
	\underset{\bomega\in \mathcal{N}}{\sup}\var\{l_{\beta_{1 k}}(\bomega)-l_{\beta_{1 k}}(\dot\bomega)\} \le L, \\&
	\underset{\bomega\in \mathcal{N}}{\sup}\var\{l_{\siga^2}(\bomega)-l_{\siga^2}(\dot\bomega)\} \le L, \quad  \underset{\bomega\in \mathcal{N}}{\sup}\var\{l_{\beta_{2r}}(\bomega) - l_{\beta_{2r}}(\dot\bomega)\} \le L,  \\& \underset{\bomega\in \mathcal{N}}{\sup}\var\{l_{\sige^2}(\bomega) - l_{\sige^2}(\dot\bomega)\}\le L, \,\,\,\, k = 1,\ldots, p_b, r=1,\ldots, p_w.
	\end{split}
	\end{equation*}
\end{lma}
\begin{proof}	
We write out $\bpsi(\bomega)- \bpsi(\dot\bomega)$, take the variance (which eliminates all no-stochastic terms) and then bound the variance uniformly on $\bomega \in \mathcal{N}$ using (\ref{bounds}) and (\ref{moments}). For example,we have
	\begin{equation*}
	\begin{split}
	l_{\beta_{0}}(\bomega)-l_{\beta_{0}}(\dot\bomega)
	&= \sumig \tau_i \sz_i^T(\dot{\bbeta}-\bbeta) % \\&
	+ \sumig (\tau_i- \dot\tau_i)(\alpha_i+\bar{e}_{i})
	\end{split}
	\end{equation*}
	so
	\begin{equation*}
	\begin{split}
	\var\{l_{\beta_{0}}(\bomega)-l_{\beta_{0}}(\dot\bomega)\}&= \sumig (\tau_i - \dot\tau_i)^2/\dot\tau_i \\&
	%
	%	=   \sumig m_i\frac{\{\dot\sigma_{e(s)}^2 -\sigma_{e(s)}^2 + m_i(\dot\sigma_{\alpha(s)}^2  - \sigma_{\alpha(s)}^2)\}^2}{(\sigma_{e(s)}^2 + m_i\sigma_{\alpha(s)}^2)^2(\dot\sigma_{e(s)}^2 + m_i\dot\sigma_{\alpha(s)}^2)}\\&
	%
	\le   2L_1^{-1}L_2^2M^2 \sumig (m_i^{-2}n^{-1}+ g^{-1}) \le 4L_1^{-3}M^2.
	\end{split}
	\end{equation*}
The argument for each of the remaining terms is similar \end{proof}

%%%%%%%%%%%%%%%%%%%%%%%%%%%%%%%%%%%%%%%%%%%%%%%%%%%%%%

\subsection{Lemmas for the derivative of $\bpsi$} \label{sec:part2}

A key part of the proof of Theorem 1 is using the mean value theorem to obtain a linear approximation for $\bpsi(\bomega)$.  We apply the mean value theorem to each (real) element of ${\bpsi}(\bomega)$ so we need to allow different arguments (i.e. values of $\bomega$) in each row of the derivative matrix.  Let $\bOmega$ be a $(p_b+p_w+3)\times (p_b+p_w+3)$ matrix and write $\nabla{\bpsi}(\bOmega)$ to mean that each row of the derivative $\nabla{\bpsi}$ is evaluated at the corresponding row of $\bOmega$.   %When the rows of $\bOmega$ all equal $\bomega^T$, we simplify the notation by replacing $\bOmega$ and its submatrices by $\bomega$.  (We discard the transpose because there is no ambiguity in doing so and the notation looks unnecessarily complicated when it is retained.)  Again discarding the transpose, we also use $\bomega$ as a generic symbol to represent any of the rows of $\bOmega$ when the specific choice of row is not important.
We also partition $\nabla{\bpsi}(\bOmega)$ into submatrices conformably with the between cluster and within cluster parameters.  Leting $ \bOmega^{(b)}$ contain the first $p_b+2$ rows and $ \bOmega^{(w)}$ the remaining $p_w+1$ rows of $\bOmega$, we can write
%
%Let $ \bOmega^{(b)}$ be a $(p_b+2)\times (p_b+p_w+3)$ matrix, let $ \bOmega^{(w)}$ be the $(p_w+1)\times (p_b+p_w+3)$ matrix and set $\bOmega$ to be the $(p_b+p_w+3)\times (p_b+p_w+3)$ matrix obtained by stacking $\bOmega^{(b)}$ above $\bOmega^{(w)}$.
%Then we write the derivative of ${\bpsi}(\bomega)$ with respect to $\bomega$ in terms of the derivatives of
%${\bpsibb}(\bomega)$ and ${\bpsiww}(\bomega)$ with respect to $\bomega^{(b)}$ and $\bomega^{(w)}$, so we have
\begin{displaymath}
\begin{split} 
\nabla{\bpsi}(\bOmega)=\left[ \begin{matrix}
\nabla{\bpsi^{(bb)}}({\bOmega}^{(b)}) &\nabla{\bpsi}^{(bw)}({\bOmega}^{(b)}) \\  \nabla{\bpsi}^{(wb)}({\bOmega}^{(w)}) & \nabla{\bpsi}^{(ww)}({\bOmega}^{(w)})
\end{matrix}\right]. %=\left[ \begin{matrix}
%\sumig\nabla\bpsibb_i\{\bomega_{(s)}\}\\
%\sumig\sumjmi\nabla\bpsiww_{ij}\{\bomega_{(s)}\}
%\end{matrix}\right].
\end{split}
\end{displaymath}
The arguments to $\nabla\bpsi^{(wb)}$ and $\nabla\bpsi^{(bw)}$ are potentially different but, when they are the same, these matrices are the transposes of each other.  When the rows of $\bOmega$ all equal $\bomega^T$, we simplify the notation by replacing $\bOmega$ and its submatrices by $\bomega$.  (We discard the transpose because there is no ambiguity in doing so and the notation looks unnecessarily complicated when it is retained.)  Again discarding the transpose, we also use $\bomega$ as a generic symbol to represent any of the rows of $\bOmega$ when the specific choice of row is not important.  

\begin{lma}   \label{lem9}
Suppose Condition A holds. Then, as $g, m_L \to\infty$,  $\left\|\bb_n - \bb\right\|=o(1)$, where $\bb_n=-\bk^{-1/2}\E\nabla\bpsi(\dot\bomega)\bk^{-1/2}$ and
\begin{equation}\label{bbbbb}
\bb=\left[ \begin{matrix}
1/\dsiga^2& \ssc_1^T/\dsiga^2&0&\bzero_{[1:p_w]}&0\\
\ssc_1/\dsiga^2& \bc_2/\dsiga^2&\bzero_{[p_b:1]}&\bzero_{[p_b:p_w]}&\bzero_{[p_b:1]}\\
0&\bzero_{[1:p_b]}& 1/(2\dsiga^4)&\bzero_{[1:p_w]}&0\\
\bzero_{[p_w:1]}&\bzero_{[p_w:p_b]}&\bzero_{[p_w:1]}&\bc_{3}/\dsige^2&\bzero_{[p_w:1]}\\
0&\bzero_{[1:p_b]}&0&\bzero_{[1:p_w]}& 1/(2\dsige^4)
\end{matrix}\right] .
\end{equation}
We have $\ba=\bb$ under normality, but not otherwise.
\end{lma} 
\begin{proof}
%Using (\ref{oedf1}), (\ref{oedf2}), (\ref{oedf3}), (\ref{oedf4}) and (\ref{oedf5}), 
From the expressions for the elements of $\E\nabla{\bpsi}(\bOmega)$ given in the Appendix, we have
\begin{equation} \label{bfin}
\begin{split}
&%\bb_{n}\\&=
\bb_n =\\&
\left[ \begin{matrix}
g^{-1}\sumig\dot\tau_i&g^{-1}\sumig\dot\tau_i \sx_i^{(b)T}&0&
\ssf^T
&0\\
g^{-1}\sumig \dot\tau_i \sx_i^{(b)}&g^{-1}\sumig \dot\tau_i\sx_i^{(b)}\sx_i^{(b)T}&\bzero_{[p_b:1]}&\bh&\bzero_{[p_b:1]}\\
0&\bzero_{[1:p_b]}&(2g)^{-1}\sumig \dot\tau_i^2 &\bzero_{[1:p_w]}&r\\
\ssf &\bh^T&\bzero_{[p_w:1]}&\bs_{w}^x/n\dsige^2+\bp&\bzero_{[p_w:1]}\\
0&\bzero_{[1:p_b]}& r &\bzero_{[1:p_w]}&q
\end{matrix}\right],
\end{split}
\end{equation}
where $\ssf=(gn)^{-1/2}\sumig \dot\tau_i \bar{\sx}_{i}^{(w)T}$, $\bh=(gn)^{-1/2}\sumig \dot\tau_i\sx_{i}^{(b)}\bar{\sx}_{i}^{(w)T}$, \\ 
$\bp=n^{-1}\sumig \dot\tau_i\bar{\sx}_{i}^{(w)}\sxbisw$ and $q=(2n)^{-1}\sumig m_i^{-2}\dot\tau_i^2+(n-g)/(2n\dsige^4)$ and $r=(4gn)^{-1/2}\sumig m_i^{-1}\dot\tau_i^2$.
It is straightforward to show from Conditions A3-A4 and the fact that $\dot\tau_i \rightarrow 1/\dot\sigma_{\alpha}^2$ uniformly in $1\le i \le g$ as $m_L \to \infty$ that  
\begin{displaymath}
\begin{split}
|g^{-1}\sumig \dot\tau_i\bar{x}_{ik}^{(b)} - c_{1k}/\dot\sigma_{\alpha}^2| %&=
%|g^{-1}\sumig\dot\tau_i\bar{x}_{ik}^{(b)}- g^{-1}\sumig \bar{x}_{ik}^{(b)}/\dsiga^2+g^{-1}\sumig \bar{x}_{ik}^{(b)}/\dsiga^2 - c_{1k}\dot\sigma_{\alpha}^2|\\&
&\le \max_{1\le i \le g}|\dot\tau_i-1/\dsiga^2|g^{-1}\sumig|\bar{x}_{ik}^{(b)}|\\\quad &+|g^{-1}\sumig\bar{x}_{ik}^{(w)} - c_{1k}|/\dsiga^2 =o(1).
\end{split}
\end{displaymath}
Similar arguments can be applied to establish the convergence of the terms the remaining terms in $-g^{-1}\E\nabla{\bpsi^{(bb)}}({\bOmega}^{(b)})$ and $-n^{-1}\E\nabla{\bpsi}^{(ww)}({\bOmega}^{(w)})$. Finally, similar arguments can be used to show that $(n/g)^{1/2}$ times the entries in the off-diagonal blocks $(ng)^{-1/2}\E\nabla{\bpsi}^{(bw)}({\bOmega}^{(b)})$ and $(ng)^{-1/2}\E \nabla{\bpsi}^{(wb)}({\bOmega}^{(w)})$ converge and then using the fact that $g=o(n)$ to show that the entries in the off-diagonal blocks converge to zero. 
%This implies that the elements of $(ng)^{-1/2}\nabla{\bpsi}^{(wb)}({\bOmega}^{(w)})$ also all converge to zero.  
 \end{proof}  

The convergence result for the expected derivative of the estimating equation that we require in order to handle $T_2(\bomega)$ is established in Lemma \ref{lem5}.   

\begin{lma}   \label{lem5}
	Suppose Condition A holds. Then, as $g, m_L \to\infty$, 
	\begin{equation*}\label{meqqp2}
 	\underset{\bomega \in \mathcal{N}}{\sup} \left\|\bk^{-1/2}
	\left\{\E\nabla\bpsi(\bOmega)-\E\nabla\bpsi(\dot\bomega)\right\}\bk^{-1/2}\right\|=o(1).
	\end{equation*}	
\end{lma}
\begin{proof}
It is enough to show the uniform convergence to zero of the elements of 
$g^{-1} \{\E\nabla\bpsi^{(bb)}(\bOmega^{(b)})-\E\nabla\bpsi^{(bb)}(\dot\bomega)\}$, 
 $g^{-1/2}n^{-1/2} \{\E\nabla\bpsi^{(bw)}(\bOmega^{(b)})-\E\nabla\bpsi^{(bw)}(\dot\bomega)\}$
%	+  \underset{\bomega \in \mathcal{N}}{\sup} $n^{-1/2}g^{-1/2}\left\|\E\nabla\bpsi^{(wb)}(\bOmega^{(w)})-\E\nabla\bpsi^{(wb)}(\dot\bomega)\right\|$
and $n^{-1}\{\E\nabla\bpsi^{(ww)}(\bOmega^{(w)})-\E\nabla\bpsi^{(ww)}(\dot\bomega)\}$.
These are all deterministic matrices so the result is obtained by directly bounding the components of these matrices. In addition to the bounds (\ref{bounds}), we also use the fact that, uniformly in $\bomega \in \mathcal{N}$, $|\sz_i^T(\dot{\bbeta}-\bbeta)|%\le M\{g^{-1/2}(1 + |\sx_i^{(b)}|) +n^{-1/2}|\bar{\sx}^{(w)}_{i}|\}
\le M g^{-1/2}\{(1 + |\sx_i^{(b)}|) +(g/n)^{1/2}|\bar{\sx}^{(w)}_{i}|\}$ to obtain bounds of the form
\[
|g^{-1}\sumig \tau_i^2\sz_i^T(\dot{\bbeta}-\bbeta)| \le L_2^{2}M g^{-1}\sumig g^{-1/2}\{(1 + |\sx_i^{(b)}|) +(g/n)^{1/2}|\bar{\sx}^{(w)}_{i}|\}= O(g^{-1/2}).
\]
%\[
%|  g^{-1/2}n^{-1/2}\sumig   \tau_i^2 \bar{\sx}_{i}^{(w)} \sz_i^T(\dot{\bbeta}-\bbeta)| %\\&
%\le  L_2^{2}  g^{-1/2}n^{-1/2}\sumig  |\bar{\sx}_{i}^{(w)}| g^{-1/2}\{(1+|\sx_{i}^{(b)}|)+(g/n)^{1/2}|\bar{\sx}_{i}^{(w)}|\} 
%=O(n^{-1/2}).
%\]
Combining these bounds, we can show that, uniformly in $\bomega \in \mathcal{N}$, $g^{-1} \{\E\nabla\bpsi^{(bb)}(\bOmega^{(b)})-\E\nabla\bpsi^{(bb)}(\dot\bomega)\} = O(g^{-1/2})$, $g^{-1/2}n^{-1/2} \{\E\nabla\bpsi^{(bw)}(\bOmega^{(b)})-\E\nabla\bpsi^{(bw)}(\dot\bomega)\}= O(n^{-1/2})$ and $n^{-1}\{\E\nabla\bpsi^{(ww)}(\bOmega^{(w)})-\E\nabla\bpsi^{(ww)}(\dot\bomega)\} = O(n^{-1/2})$ and the result follows. \end{proof}

The final result we require  in order to handle $T_3(\bomega)$  and complete the proof of Theorem 1 is given in Lemma \ref{lem8}.

\begin{lma}  \label{lem8}
	Suppose Condition A holds.    As $g, m_L\to\infty$,
	\begin{equation*}
	\begin{split}
	& \underset{\bomega\in\mathcal{N}}{\sup} g^{-1/4}\|\bkmh\{\nabla\bpsi(\bOmega)-\E\nabla\bpsi(\bOmega)\}\bkmh\| = o_p(1).
	\end{split}
	\end{equation*}	
\end{lma}
\begin{proof}  
Arguing as in the proof of Lemma \ref{lem5}, it is enough to show the uniform convergence to zero of the elements of 
$g^{-5/4} \{\nabla\bpsi^{(bb)}(\bOmega^{(b)})-\E\nabla\bpsi^{(bb)}(\bOmega^{(b)})\}$, 
 $g^{-3/4}n^{-1/2} \{\nabla\bpsi^{(bw)}(\bOmega^{(b)})-\E\nabla\bpsi^{(bw)}(\bOmega^{(b)})\}$
%	+  \underset{\bomega \in \mathcal{N}}{\sup} $n^{-1/2}g^{-1/2}\left\|\E\nabla\bpsi^{(wb)}(\bOmega^{(w)})-\E\nabla\bpsi^{(wb)}(\dot\bomega)\right\|$
and $g^{-1/4}n^{-1}\{\nabla\bpsi^{(ww)}(\bOmega^{(w)})-\E\nabla\bpsi^{(ww)}(\bOmega^{(w)})\}$.
We use the bounds (\ref{bounds})  and the fact that,
by direct calculation of means and variances, we have
	\begin{equation*}
	\begin{split}
	&\sumig(1+|\sx_i^{(b)}|)|\alpha_i + \bar{e}_{i}| = O_p(g), \quad \sumig|\bar{\sx}_i^{(w)}||\alpha_i + \bar{e}_{i}| = O_p(g) \quad \mbox{and} \\&\sumig |(\alpha_i + \bar{e}_{i})^2 - \dsiga^2 - m_i^{-1}\dsige^2| = O_p(g).
	\end{split}
	\end{equation*}
For the derivatives with respect to the variance components, we have
%	
%	The non-zero elements in the  $(p_b+2) \times (p_b+2)$ matrix $\nabla{\bpsi}^{(bb)}({\bOmega}^{(b)})-\E \nabla{\bpsi}^{(bb)}({\bOmega}^{(b)})$ are of the form
%	\begin{equation*}
%	\begin{split}
%	|l_{\beta_{0}\siga^2}(\bomega) - \E l_{\beta_{0}\siga^2}(\bomega)| \le   L_2^{2}\sumig |\alpha_i + \bar{e}_{i}| =O_p(g),\\
%	%
%	| l_{\beta_{1k}\siga^2}(\bomega)  - \E l_{\beta_{1k}\siga^2}(\bomega)| \le L_2^{2}  \sumig |\sx_{ik}^{(b)}||\alpha_i + \bar{e}_{i}|=O_p(g),
%	\end{split}
%	\end{equation*}
%	%
	\begin{equation*} 
	\begin{split}
	|l_{\siga^2\siga^2}(\bomega) -& \E l_{\siga^2\siga^2}(\bomega)| \le  L_2^{3}\sumig  |(\alpha_i + \bar{e}_{i})^2 -  \dsiga^2 - m_i^{-1}\dsige^2| \\&
	+ 2L_2^{3}M  \sumig  \{g^{-1/2}(1+ |\sx_i^{(b)}|) + n^{-1/2}|\bar{\sx}_{i}^{(w)}|\}|\alpha_i + \bar{e}_{i}| =O_p(g),
	\end{split}
	\end{equation*}
%
%	The non-zero elements in the $(p_b+2) \times (p_w+1)$ matrix $\nabla{\bpsi}^{(bw)}({\bOmega}^{(b)})-\E \nabla{\bpsi}^{(bw)}({\bOmega}^{(b)})$ are 
	\begin{equation*}
	\begin{split}
%	|l_{\beta_{0}\sige^2}(\bomega) - \E l_{\beta_{0}\sige^2}(\bomega)|&\le m_L^{-1}L_2^{-2} \sumig |\alpha_i + \bar{e}_{i}| = O_p(m_L^{-1}g),\\
%	%
%	|l_{\beta_{1k}\sige^2}(\bomega) - \E l_{\beta_{1k}\sige^2}(\bomega)| & = m_L^{-1}L_1^{-2}\sumig |\sx_{ik}^{(b)}||\alpha_i + \bar{e}_{i}| = O_p(m_L^{-1}g),\\
%		%
%	|l_{\siga^2\beta_{2r}}(\bomega)- \E l_{\siga^2\beta_{2r}}(\bomega)| &\le L_2^{2}  \sumig |\bar{\sx}_{i}^{(w)}| |\alpha_i + \bar{e}_{i}| = O_p(g) \\
%	%
	|l_{\siga^2\sige^2}(\bomega)-& \E l_{\siga^2\sige^2}(\bomega)| \le m_L^{-1}L_2^{3}  \sumig |(\alpha_i + \bar{e}_{i})^2 -  \dsiga^2 - m_i^{-1}\dsige^2|  \\&
	+ m_L^{-1}2L_2^{3}M \sumig \{g^{-1/2}(1+|\sx_i^{(b)}|)+n^{-1/2}|\bar{\sx}_{i}^{(w)}|\}|\alpha_i + \bar{e}_{i}| = O_p(m_L^{-1}g).
	\end{split}
	\end{equation*}
%	
%	
%	For the first $p_w$ rows in the $(p_w+1) \times (p_w+1)$ matrix $\nabla{\bpsi}^{(ww)}({\bOmega}^{(w)})-\E \nabla{\bpsi}^{(w)}({\bOmega}^{(ww)})$, we have 
%	\begin{equation*}
%	\begin{split}
%	l_{\bbeta_{2k}\beta_{2r}}(\bomega) - \E l_{\beta_{2k}\beta_{2r}}(\bomega)& = 0,\\
%	%
%	|l_{\beta_{2k}\sige^2}(\bomega) - \E l_{\beta_{2k}\sige^2}(\bomega)| &= | \frac{1}{\sige^4}(S_{wk}^{xy}-\bs_{wk}^{xT}\dot\bbeta_{2}) + \sumig m_i^{-1}\tau_i^2 \bar{\sx}_{ik}^{(w)}(\alpha_i+\bar{e}_{i}) |\\
%	&\le | \frac{1}{\sige^4}\sum_{i=1}^g\sum_{j=1}^{m_i} (x_{ijk}^{(w)}-\bar{x}_{ik}^{(w)})e_{ij}| + m_L^{-1}L_2^{2}\sumig |\alpha_i+\bar{e}_{i} | =O_p(n);
%	\end{split}
%	\end{equation*}
	and
	\begin{equation*}
	\begin{split}
	|l_{\sige^2\sige^2}(\bomega) & - \E l_{\sige^2\sige^2}(\bomega)|  \le %\sige^{-6}|S_{w}^y-2\bs_{w}^{xy}\bbeta_{2}+\bbeta_{2}^T\bs_{w}^x\bbeta_{2} -\{\dot\bbeta_{2}-\bbeta_{2}\}^T\bs_{w}^x\{\dot\bbeta_{2}-\bbeta_{2}\} - (n-g)\dsige^2|  \\&+ m_L^{-2}L_2^{3} \sumig |(\alpha_i + \bar{e}_{i})^2 - \dsiga^2 - m_i^{-1}\dsige^2|\\&
%	+ 2 m_L^{-2}L_2^{3} \sumig\{g^{-1/2}(1+|\sx_i^{(b)}|) + n^{-1/2}|\bar{x}_{ik}^{(w)}|\}|\alpha_i + \bar{e}_{i}|\\&
%	= 
	\sige^{-6}|2(\dot{\bbeta}_{2}-\bbeta_2)^T\sumig\sum_{j=1}^{m_i}(\sx_{ij}^{(w)}-\bar{\sx}_{i}^{(w)})e_{ij}
	+\sumig\sum_{j=1}^{m_i}(e_{ij}^2-\dot\sigma_e^2) \nonumber\\&
	-\sumig({m_i}\bar{e}_{i}^2 - \dot\sigma_e^2)|  + m_L^{-2}L_2^{3} \sumig |(\alpha_i + \bar{e}_{i})^2 - \dsiga^2 - m_i^{-1}\dsige^2|\\&
	+ 2 m_L^{-2}L_2^{3} \sumig\{g^{-1/2}(1+|\sx_i^{(b)}|) + n^{-1/2}|\bar{x}_{ik}^{(w)}|\}|\alpha_i + \bar{e}_{i}| = O_p(n)
	\end{split}
	\end{equation*}
	because $g < n$ implies $m_L^{-2}g^{1/2} < m_L^{-2}g < n$. 
\end{proof}

\appendix

\section{Appendix: The derivative and expected derivative of $\bpsi$}

For the first row in $\nabla{\bpsi}^{(bb)}({\bOmega}^{(b)})$, we have 
\begin{equation*}\label{df1}
\begin{split}
&l_{\beta_{0}\beta_{0}}(\bomega)= - \sumig \tau_i,\quad
\ssl_{\beta_{0}\bbeta_{1}}(\bomega)^T=  - \sumig \tau_i\sx_i^{(b)T},\quad %\\&
l_{\beta_{0}\siga^2}(\bomega) =-  \sumig \tau_i^2 (\yib-\sz_i^T\bbeta);
\end{split}
\end{equation*}
for rows $k = 2,\ldots, p_b+1$ in $\nabla{\bpsi}^{(bb)}({\bOmega}^{(b)})$, we have 
\begin{equation*}\label{df2}
\begin{split}
&l_{\beta_{1k}\beta_{0}}(\bomega)=  - \sumig \tau_i\sx_{ik}^{(b)}, \quad \ssl_{\beta_{1k}\bbeta_{1}}(\bomega)^T=-  \sumig \tau_i \sx_{ik}^{(b)}\sx_{i}^{(b)T},  \\&
l_{\beta_{1k}\siga^2}(\bomega) = -  \sumig \tau_i^2\sx_{ik}^{(b)}(\yib-\sz_i^T\bbeta);  
\end{split}
\end{equation*}
and for the $(p_b+2)$th row in $\nabla{\bpsi}^{(bb)}({\bOmega}^{(b)})$, we have 
\begin{equation*}\label{df3}
\begin{split}
& l_{\siga^2\beta_{0}}(\bomega)=-  \sumig \tau_i^2(\yib-\sz_i^T\bbeta), 
\quad %\\&
\ssl_{\siga^2 \bbeta_{1}}(\bomega)^T= -  \sumig \tau_i^2\sx_{i}^{(b)T}(\yib-\sz_i^T\bbeta),\\&
 l_{\siga^2\siga^2}(\bomega)= \frac{1}{2} \sumig \tau_i^2-  \sumig  \tau_i^3(\yib-\sz_i^T\bbeta)^2.
\end{split}
\end{equation*}
The rows of $\nabla{\bpsi}^{(bw)}({\bOmega}^{(b)})$ are
\begin{equation*}\label{df12}
\begin{split}
&
\ssl_{\beta_{0}\bbeta_{2}}(\bomega)^T=-\sumig\tau_i\bar{\sx}_{i}^{(w)T},\quad %\\&
l_{\beta_{0}\sige^2}(\bomega)= -  \sumig m_i^{-1}\tau_i^2(\yib-\sz_i^T\bbeta);
\\&
\ssl_{\bbeta_{1k}\bbeta_{2}}(\bomega)^T=-\sumig\tau_i x_{ik}^{(b)}\bar{\sx}_{i}^{(w)T}, \quad %\qquad k=1,\ldots, p_b; \\&
l_{\beta_{1k}\sige^2}(\bomega) = -  \sumig m_i^{-1}\tau_i^2\sx_{ik}^{(b)}(\yib-\sz_i^T\bbeta);
\quad \\&
\ssl_{\siga^2\bbeta_{2}}(\bomega)^T= -  \sumig \tau_i^2\sxbisw(\yib-\sz_i^T\bbeta), \quad \\&
l_{\siga^2\sige^2}(\bomega)= \frac{1}{2} \sumig m_i^{-1}\tau_i^2-  \sumig m_i^{-1}\tau_i^3(\yib-\sz_i^T\bbeta)^2;
\end{split}
\end{equation*}
$ k=1,\ldots, p_b$, and, finally, the rows of $\nabla{\bpsi}^{(ww)}({\bOmega}^{(w)})$ are 
\begin{equation*}\label{df4}
\begin{split}
&\ssl_{\beta_{2k}\bbeta_{2}}(\bomega)^T= - \frac{1}{\sige^2}\bs_{wk}^{xT}-\sumig\tau_i\bar{x}_{ik}^{(w)}\sxbisw;
\\&
l_{\beta_{2k}\sige^2}(\bomega)= - \frac{1}{\sige^4}\{S_{wk}^{xy}-\bs_{wk}^{xT}\bbeta_{2}\}
-\sumig m_i^{-1} \tau_i^2 \bar{x}_{ik}^{(w)}(\yib-\sz_i^T\bbeta);\\
&
\ssl_{\sige^2\bbeta_{2}}(\bomega)^T= - \frac{1}{\sige^4}\{\bs_{w}^{xyT}-\bbeta_{2}^T\bs_{w}^x\}
-\sumig m_i^{-1} \tau_i^2\sxbisw (\yib-\sz_i^T\bbeta),\\\
& l_{\sige^2\sige^2}(\bomega)=  \frac{1}{2} \sumig m_i^{-2} \tau_i^2+\frac{n-g}{2\sige^4}-\frac{1}{\sige^6}(S_{w}^y-2\bbeta_{2}^T\bs_{w}^{xy}+\bbeta_{2}^T\bs_{w}^x\bbeta_{2})\\&\qquad\qquad
- \sumig m_i^{-2} \tau_i^3 (\yib-\sz_i^T\bbeta)^2,
\end{split}
\end{equation*}
$k = 1,\ldots, p_w$.  Here we have written $\bs_{wk}^{xT}$ for the $k$th row of $\bs_{w}^x$ so $\bs_{w}^x = [\bs_{w1}^{x},\ldots, \bs_{wp_w}^{x}]^T$ and $x_{ik}^{(b)}$, $\bar{x}_{ik}^{(w)} $and  $S_{wk}^{xy}$ for the $k$th element of $\sx_{i}^{(b)}$, $\bar{\sx}_{i}^{(w)}$ and  $\bs_{wk}^{xy} $, respectively, so $\sx_{i}^{(b)}=[x_{ik}^{(b)}]$, $\bar{\sx}_{ik}^{(w)}=[\bar{x}_{ik}^{(w)}]$ and  $\bs_{wk}^{xy} = [S_{wk}^{xy}]$.  When we need to address the elements of $\bs_w^x$, we write $\bs_w^x = [S_{wkr}^x]$.

We calculate the expected derivative matrix using $\E(\yib-\sz_i^T\bbeta)^2 =\left\{\sz_i^T(\dot{\bbeta}-\bbeta)\right\} ^2+\dot\tau_i^{-1}$, 
$\E (\bs_{w}^{xy}) = \bs_{w}^x\bbeta_{2}$ and
%\begin{equation*}\label{dabp3}
$\E (\swy-2\bbeta_{2}^T\bs_{w}^{xy}+\bbeta_{2}^T\bs_{w}^x\bbeta_{2})=(\dot{\bbeta}_{2}-\bbeta_{2})^T\bs_{w}^x(\dot{\bbeta}_{2}-\bbeta_{2})+(n-g)\dsige^2$.
%\end{equation*}
The first row of $\E\nabla {\bpsi}^{(bb)}({\bOmega}^{(b)})$ is
\begin{equation*}\label{edf1}
\begin{split}
&\E l_{\beta_{0}\beta_{0}}(\bomega)=- \sumig \tau_i,\quad
\E \ssl_{\beta_{0}\bbeta_{1}}(\bomega)^T= -  \sumig \tau_i \sx_i^{(b)T},\\&
\E l_{\beta_{0}\siga^2}(\bomega)= -  \sumig \tau_i^2\sz_i^T(\dot{\bbeta}-\bbeta);
\end{split}
\end{equation*}
rows $k=2,\ldots,p_b+1$ of $\E\nabla {\bpsi}^{(bb)}({\bOmega}^{(b)})$ are
\begin{equation*}\label{edf2}
\begin{split}
&\E l_{\beta_{1k}\beta_{0}}(\bomega) = -  \sumig \tau_i \sx_{ik}^{(b)}, \quad
\E \ssl_{\beta_{1k}\bbeta_{1}}(\bomega)^T= -  \sumig \tau_i \sx_{ik}^{(b)}\sx_i^{(b)T},\\&
\E l_{\beta_{1k}\siga^2}(\bomega)= -  \sumig \tau_i^2\sx_{ik}^{(b)} \sz_i^T(\dot{\bbeta}-\bbeta);
\end{split}
\end{equation*}
and the $(p_b+2)$th row of  $\E\nabla {\bpsi}^{(bb)}({\bOmega}^{(b)})$ is
\begin{equation*}\label{df31}
\begin{split}
&\E  l_{\siga^2\beta_{0}}(\bomega)=-  \sumig \tau_i^2\sz_i^T(\dot{\bbeta}-\bbeta),\quad% \\&
\E \ssl_{\siga^2 \bbeta_{1}}(\bomega)^T= -  \sumig \tau_i ^2\sx_{i}^{(b)T} \sz_i^T(\dot{\bbeta}-\bbeta),\\
&
\E l_{\siga^2\siga^2}(\bomega)= \frac{1}{2} \sumig \tau_i^2(1 -  2\dot\tau_i^{-1}\tau_i) %\\&\qquad\qquad\qquad\qquad
-  \sumig \tau_i^3\{\sz_i^T(\dot{\bbeta}-\bbeta)\}^2.
\end{split}
\end{equation*}	
The first $p_w$ columns of $\E\nabla {\bpsi}^{(bw)}({\bOmega}^{(b)})$ are 
\begin{equation*}\label{edf12}
\begin{split}
&
\E \ssl_{\beta_{0}\bbeta_{2}}(\bomega)^T=-\sumig \tau_i\bar{\sx}_{i}^{(w)T},\quad %\\&
\E \ssl_{\bbeta_{1k}\bbeta_{2}}(\bomega)^T=-\sumig \tau_i x_{ik}^{(b)}\bar{\sx}_{i}^{(w)T}; \\&
\E \ssl_{\siga^2\bbeta_{2}}(\bomega)^T= -  \sumig \tau_i^2\sxbisw \sz_i^T(\dot{\bbeta}-\bbeta),
\end{split}
\end{equation*}
$k=1,\ldots, p_b$.  The last column of $\E\nabla {\bpsi}^{(bw)}({\bOmega})$  is
\begin{equation*}\label{edf13}
\begin{split}
&
\E  l_{\beta_{0}\sige^2}(\bomega)=-  \sumig m_i^{-1}\tau_i^2\sz_i^T(\dot{\bbeta}-\bbeta);\\&
\E l_{\bbeta_{1k}\sige^2}(\bomega)= -  \sumig m_i^{-1}\tau_i^2 x_{ik}^{(b)}\sz_i^T(\dot{\bbeta}-\bbeta) ;\\&
\E l_{\siga^2\sige^2}(\bomega)= \frac{1}{2} \sumig m_i^{-1}\tau_i^2 (1 -  2\dot\tau_i^{-1}\tau_i) % \\&\qquad\qquad\qquad\qquad
-  \sumig m_i^{-1}\tau_i^3\left\{\sz_i^T(\dot{\bbeta}-\bbeta)\right\}^2,
\end{split}
\end{equation*}	
$k=1,\ldots,p_b$, and, finally, the rows of $\E\nabla{\bpsi}^{(ww)}({\bOmega})$ are 
\begin{equation*}\label{df41}
\begin{split}
&\E\ssl_{\beta_{2k}\bbeta_{2}}(\bomega)^T= -\frac{1}{\sige^2}\bs_{wk}^{xT}-\sumig\tau_i\bar{x}_{ik}^{(w)}\sxbisw,\,\,\,\, k = 1,\ldots, p_w;
\\&
\E l_{\beta_{2k}\sige^2}(\bomega)= - \frac{1}{\sige^4}\bs_{wk}^{xT}(\dot{\bbeta}_2-\bbeta_{2})-\sumig m_i^{-1}\tau_i^2\bar{x}_{ik}^{(w)}\sz_i^T(\dot{\bbeta}-\bbeta)
,\,\,\,\, k = 1,\ldots, p_w;
\\
&
\E\ssl_{\sige^2\bbeta_{2}}(\bomega)^T= - \frac{1}{\sige^4}(\dot{\bbeta}_{2}-\bbeta_{2})^T\bs_{w}^x
-\sumig m_i^{-1}\tau_i^2\bar{\sx}_{ik}^{(w)T}\sz_i^T(\dot{\bbeta}-\bbeta)
\\
& \E l_{\sige^2\sige^2}(\bomega) = \frac{1}{2}\sumig m_i^{-1}\tau_i^2(1-2 \dot\tau_i^{-1}\tau_i)+ \frac{n-g}{2\sige^4}\big(1  -2\frac{\dot\sigma_{e}^2}{\sige^2}\big)
\\&\qquad\qquad\qquad\qquad
 -\frac{1}{\sige^6}(\dot{\bbeta}_{2}-\bbeta_{2})^T\bs_{w}^x(\dot{\bbeta}_{2}-\bbeta_{2})
- \sumig m_i^{-2}\tau_i^3 \{\sz_i^T(\dot{\bbeta}-\bbeta) \}^2.
\end{split}
\end{equation*}

%%%%%%%%%%%%%%%%%%%%%%%%%%%%%%%%%%%%%%%%%%%%%%%%%%%%%

%\section*{References}
\bibliographystyle{plainnat}
\bibliography{LyuWelsh}

\end{document}